\documentclass[10pt,oneside,english]{amsart}
\usepackage[T1]{fontenc}
\usepackage[latin9]{inputenc}
\usepackage[a4paper]{geometry}
\usepackage{color}
\usepackage{babel}
\usepackage{booktabs}
\usepackage{units}
\usepackage{textcomp}
\usepackage{amstext}
\usepackage{amsthm}
\usepackage{amssymb}
\usepackage{graphicx}
\usepackage{wasysym}
\usepackage[unicode=true,pdfusetitle,
 bookmarks=true,bookmarksnumbered=false,bookmarksopen=false,
 breaklinks=false,pdfborder={0 0 0},pdfborderstyle={},backref=false,colorlinks=true]
 {hyperref}
\hypersetup{
 linkcolor=blue,  citecolor=blue, urlcolor=blue}

\makeatletter

\newcommand{\noun}[1]{\textsc{#1}}
\providecommand{\tabularnewline}{\\}
\usepackage{tikz}
\usetikzlibrary{calc}

\numberwithin{equation}{section}
\numberwithin{figure}{section}
\numberwithin{table}{section}
\theoremstyle{plain}
\newtheorem*{thm*}{\protect\theoremname}
\theoremstyle{plain}
\newtheorem*{prop*}{\protect\propositionname}
\theoremstyle{remark}
\newtheorem*{acknowledgement*}{\protect\acknowledgementname}
\theoremstyle{plain}
\newtheorem{thm}{\protect\theoremname}[section]
\theoremstyle{definition}
\newtheorem{defn}[thm]{\protect\definitionname}
\theoremstyle{remark}
\newtheorem{rem}[thm]{\protect\remarkname}
\theoremstyle{definition}
\newtheorem{example}[thm]{\protect\examplename}
\theoremstyle{plain}
\newtheorem{fact}[thm]{\protect\factname}
\theoremstyle{plain}
\newtheorem{lem}[thm]{\protect\lemmaname}
\theoremstyle{plain}
\newtheorem{prop}[thm]{\protect\propositionname}
\theoremstyle{plain}
\newtheorem{cor}[thm]{\protect\corollaryname}

\@ifundefined{date}{}{\date{}}
\usepackage[abbrev]{amsrefs}
\usepackage{graphicx}
\usepackage{setspace}
\usepackage{url}
\usepackage[perpage,symbol]{footmisc}
\usepackage{multicol}
\usepackage{etoolbox}
\usepackage{bbm}
\usepackage{enumitem}
\usepackage{caption}
\usepackage{etoolbox}
\usepackage{stackrel}

\setlist[enumerate]{leftmargin=*,widest=0}
\setitemize[0]{leftmargin=10pt,itemindent=0pt}
\setlist{nosep}

\DefineFNsymbols*{lamportnostar}[math]{{(\dagger)}{(\ddagger)}{(\S)}{(\P)}{(\|)}{(\dagger\dagger)}{(\ddagger\ddagger)}}
\setfnsymbol{lamportnostar}

\DeclareMathOperator{\diag}{diag}

\DeclareMathOperator{\Spec}{Spec}
\DeclareMathOperator{\ord}{ord}

\DeclareMathOperator{\rank}{rank}

\DeclareMathOperator{\col}{col}

\newcommand{\one}{\mathbbm{1}}
\newcommand{\At}{\scalebox{0.9}{\ensuremath{\widetilde{A}}}}
\newcommand{\Mod}[1]{\,\left(\textup{mod}\;#1\right)}

\usepackage{babel}
\providecommand{\definitionname}{Definition}
\providecommand{\propositionname}{Proposition}
\providecommand{\theoremname}{Theorem}
\theoremstyle{plain}

\theoremstyle{definition}
\theoremstyle{definition}

\newtheorem{rem}[thm]{Remark}
\newtheorem*{rems*}{Remarks}
\newtheorem*{disc*}{Discussion}

\theoremstyle{plain}
\newtheorem{claim}[thm]{\protect\claimname}

\makeatother

\providecommand{\acknowledgementname}{Acknowledgement}
\providecommand{\corollaryname}{Corollary}
\providecommand{\definitionname}{Definition}
\providecommand{\examplename}{Example}
\providecommand{\factname}{Fact}
\providecommand{\lemmaname}{Lemma}
\providecommand{\propositionname}{Proposition}
\providecommand{\remarkname}{Remark}
\providecommand{\theoremname}{Theorem}

\begin{document}
\title[Free flags and powering of HD-expanders]{\vspace*{-1cm}
Free flags over local rings and powering\\
of high dimensional expanders}
\author{Tali Kaufman and Ori Parzanchevski}
\date{\today}
\begin{abstract}
Powering the adjacency matrix of an expander graph results in a better
expander of higher degree. In this paper we seek an analogue operation
for high-dimensional expanders. We show that the naive approach to
powering does not preserve high-dimensional expansion, and define
a new power operation, using geodesic walks on quotients of Bruhat-Tits
buildings. Applying this operation results in high-dimensional expanders
of higher degrees. The crux of the proof is a combinatorial study
of flags of free modules over finite local rings. Their geometry describes
links in the power complex, and showing that they are excellent expanders
implies high dimensional expansion for the power-complex by Garland's
local-to-global technique. As an application, we use our power operation
to obtain new efficient double samplers.
\end{abstract}

\maketitle

\section{Introduction}

A $k$-regular graph is called an \emph{expander }if the nontrivial
eigenvalues of its adjacency matrix are of small magnitude in comparison
with $k$ (the trivial eigenvalues are, by definition, $\pm k$).
The theory of expanders has long been a fruitful meeting point for
combinatorics, algebra, number theory and computer science, see e.g.\ the
surveys \cite{HLW06,Lub12}. In recent years a theory of high-dimensional
(HD) expanders has emerged, and is already seeing applications in
mathematics and computer science, e.g.\ in PCPs \cite{Dinur2017Highdimensionalexpanders},
property testing \cite{Kaufman2014HDTesting}, expansion in finite
groups \cite{Chapman2019CutoffRamanujancomplexes}, quantum computation
\cite{Evra2018RamanujancomplexesGolden}, counting problems in matroids
\cite{Anari2018matroids}, list decoding \cite{Dinur2019listdecoding}
and lattices \cite{Kaufman2018lattices}; we refer the reader to \cite{lubotzky2017high}
for a recent survey.

A major obstacle in the study of HD expanders is that the means available
for constructing simplicial complexes are much scarcer in dimension
greater than one; for example, there are many well understood models
of random graphs, whereas the theory of random complexes is still
in its infancy. In this paper we focus on the operation of \emph{powering}:
given a graph $\mathcal{G}$ with adjacency matrix $A=A_{\mathcal{G}}$,
one can regard the matrix $A^{r}$ as the adjacency matrix of a ``power
graph'' $\mathcal{G}^{r}$. If $\mathcal{G}$ is $k$-regular with
second largest eigenvalue (in absolute value) $\lambda$, then $\mathcal{G}^{r}$
is a $k^{r}$-regular graph with second largest eigenvalue $\lambda^{r}$.
On average, distances between vertices become shorter in the power
graph, and if the original graph was an expander, the power graph
is a better one (namely, the ratio between first and second eigenvalues
improves). Powering of graphs has been used, for example, in the proof
of the PCP Theorem by Dinur \cite{Dinur2007PCPtheoremgap} and in
the Zig-Zag construction \cite{reingold2002entropy}, and various
researchers have independently raised the question whether there exists
a power operation for high dimensional expanders.

Unlike the case of graphs, there are many non-equivalent definitions
for HD expansion (e.g., \cite{friedman1995second,LinialMeshulam2006,Gro10,FGL+11,dotterrer2010coboundary,mukherjee2012cheeger,parzanchevski2012isoperimetric,gundert2016eigenvalues,first2016ramanujan,Parzanchevski2012walks,Parzanchevski2013,Kaufman2017Colorful,Oppenheim2017Localspectral,evra2015systolic}),
reflecting the richness of high-dimensional combinatorics. A precursor
of the modern theory of HD expanders is Garland's seminal paper \cite{Gar73},
which introduces a local-to-global approach: it shows that a $d$-dimensional
simplicial complex, whose one-dimensional links are (very good) expander
graphs, is cohomologically connected in all dimensions between zero
and $d$. The complexes Garland was interested in are quotients of
Bruhat-Tits buildings (see §\ref{subsec:-Buildings-and-complexes}).
The links of these complexes are flag complexes over finite fields,
which are well known to be excellent expanders. Currently, we know
that expansion in links implies several global expansion properties
\cite{cohen2014inverse,gundert2016eigenvalues,Oppenheim2017Localspectral,Kaufman2017Highorderrandom},
and in this paper we define a HD expander as a complex whose links
are good expanders (see Definition \ref{def:HDX}).\medskip{}

It turns out that the most natural approach to powering (which is
described in §\ref{subsec:Spheres-and-natural}) does not preserves
high dimensional expansion, so that a more sophisticated one has to
be taken. We devise a powering method based on the notion of geodesic
paths from \cite{Lubetzky2017RandomWalks}; as an example, if $X$
is a two-dimensional complex, its \emph{geodesic $r$-power} is the
clique complex with the same vertices as $X$, and an edge between
every two vertices which were connected by a monochromatic geodesic
path of length $r$ in $X$. The rigorous definition appears in §\ref{subsec:Geodesic-powering},
after we recall the definitions of simplicial complexes (§\ref{subsec:Simplicial-complexes})
and buildings (§\ref{subsec:-Buildings-and-complexes}). An explicit
example of a geodesic $r$-power Cayley complex appears in §\ref{subsec:Explicit-example}.\medskip{}

The central part of the paper is §\ref{subsec:Spectrum-of-links},
which shows that when applied to $\At$\emph{-complexes}, namely,
quotients of affine Bruhat-Tits buildings of type $\At$, our geodesic
power operation yields HD expanders. Let us explain where this expansion
comes from: Each link in an $\At$-complex is a spherical building
over a finite field, and its cells correspond to flags in a finite
vector space. The one-dimensional links are either complete bipartite
graphs, or the projective plane $\mathbb{P}^{2}\mathbb{F}_{q}$, which
is the incidence graph of lines and planes in $\mathbb{F}_{q}^{3}$.
It is a classic exercise that the adjacency spectrum of $\mathbb{P}^{2}\mathbb{F}_{q}$
is $\{\pm(q+1),\pm\sqrt{q}\}$, making it an excellent expander, as
$\frac{\sqrt{q}}{q+1}$ can be made arbitrarily small. This observation
is the basis for Garland's work, and our main task is to conduct a
parallel study for links in our power-complex. In Proposition \ref{prop:link-power}
we show that these links correspond to flags of free submodules of
a fixed free module over a finite local ring $R$, such as $R=\mathbb{Z}/p^{r}\mathbb{Z}$.
We call the complex which arises in this manner the \emph{free projective
space over $R$} (see Definition \ref{def:free_proj}), and we point
out that this object may be of independent interest, outside the realm
of buildings and HD expanders. Our main achievement is a complete
analysis of the spectrum of one-dimensional links:
\begin{thm*}[Main theorem]
For $d\geq2$, the links of $(d-2)$-cells in the geodesic $r$-power
of an $\At$-complex of dimension $d$ are either complete bipartite
graphs, or expanders with spectrum 
\[
\left\{ \pm\left(q+1\right)q^{r-1},\pm\sqrt{q^{2r-1}},\pm\sqrt{q^{2r-2}},\ldots,\pm\sqrt{q^{r+1}},\pm\sqrt{q^{r}}\right\} .
\]
Consequently, geodesic powers of $\At$-complexes are high-dimensional
expanders.
\end{thm*}
The computation of the spectrum of these links is carried out in Theorem
\ref{thm:spec-r-links}, whose proof is long and technically challenging,
in comparison with the elegance of the final result. This, together
with the fact that they arise in a natural algebraic settings, suggest
that a broader geometric theory of free flags over finite rings could
perhaps be developed to give a more conceptual proof. Another interesting
corollary of Theorem \ref{thm:spec-r-links} is an isospectrality
result (Corollary \ref{cor:isospect}) for free projective planes
over local rings with the same residue order (or equivalently, of
links of geodesic powers of $\At$-complexes of the same densities).\medskip{}

Having established local HD expansion in §\ref{subsec:Spectrum-of-links},
we turn in §\ref{subsec:Expansion-between-vertices-and-geod} to demonstrating
expansion between vertices in the power-complex (Theorem \ref{thm:spec-geodesics}),
and between vertices and geodesics (Proposition \ref{prop:vert-vs-geod})
-- these are required for the applications which we present in §\ref{sec:Applications-to-sampling}.
We stress that while the notion of geodesic walks comes from \cite{Lubetzky2017RandomWalks},
the expansion types studied here and there are different, and no result
from that paper is used in this one.\medskip{}

A special family of $\At$-complexes which appears in the second half
of the paper are \emph{Ramanujan complexes}: Ramanujan graphs, which
were defined in \cite{LPS88}, are $k$-regular graphs whose nontrivial
eigenvalues belong to the $L^{2}$-spectrum of the $k$-regular tree.
As regular trees are one dimensional buildings, Ramanujan complexes
were defined in \cite{li2004ramanujan,Lubotzky2005a} to be $\At$-complexes
whose spectral theory mimics that of the $\At$-building (see Definition
in §\ref{subsec:-Buildings-and-complexes}). It turns out, however,
that in dimension two and above \emph{all }quotients of buildings
have some expansion properties (by Garland or by Kazhdan's property
(T)); while in contrast, \emph{every} regular graph is a quotient
of a one dimensional building. Inspection reveals that many results
on Ramanujan complexes actually apply to general quotients of HD buildings,
e.g.\ \cite{FGL+11,evra2014mixing,evra2015finite,kaufman2016isoperimetric,Golubev2013triangle}
(some results which use the full power of the Ramanujan property appear
in \cite{Evra2018RamanujancomplexesGolden,Chapman2019CutoffRamanujancomplexes,Lubetzky2017RandomWalks}).
In this paper, the results of section §\ref{subsec:Spectrum-of-links}
apply to all $\At$-complexes, but in §\ref{subsec:Expansion-between-vertices-and-geod}
we restrict ourselves to Ramanujan complexes in order to obtain a
stronger result. 

While we focus in this paper on $\At$-complexes, our power operation
makes sense for any colored complex (see Definition \ref{def-colored-clique-complexes}),
and it is plausible that it yields HD expanders from other ones as
well: natural candidates are spherical buildings \cite{lubotzky2014expansion},
the random ones constructed in \cite{Lubotzky2015RandomLatinsquares,Lubotzky2018RandomSteinersystems},
and the ones constructed in \cite{Kaufman2018Construction}. 

\subsection{\label{subsec:Spheres-and-natural}Spheres and natural powering}

When considering clique complexes, a natural power operation comes
to mind: taking the clique complex afforded by the $r$-power of the
one-skeleton of the original complex. In effect, this is not so simple,
since the $r$-power of a graph is a ``multigraph'' with multiple
edges and loops, and it is not clear how to define the clique complex
in this case. Indeed, in the power graph $\mathcal{G}^{r}$ two vertices
are neighbors if there is a path of length $r$ between them in $\mathcal{G}$,
and backtracking paths always give rise to loops. One can replace
$\mathcal{G}^{r}$ by the ``non-backtracking $r$-power'' $\mathcal{G}^{[r]}$,
in which two vertices are neighbors if there is a \emph{non-backtracking}
$r$-path between them in $\mathcal{G}$. The expansion quality of
$\mathcal{G}^{[r]}$ is even better than that of $\mathcal{G}^{r}$
\cite{Alon2007Nonbacktrackingrandom}, and if $\mathrm{girth}(\mathcal{G})>r$
then $\mathcal{G}^{[r]}$ is a simple graph (a graph with no multiedges
and loops). However, for HD expanders this is still not useful, since
the girth of the one-skeleton is only three, and more edges of $\mathcal{G}^{r}$
should be removed to obtain a simple graph. For $\mathcal{G}$ of
high girth, the vertices of $\mathcal{G}^{[r]}$ correspond to the
$r$-sphere in $\mathcal{G}$, which leads us to observe the $r$-sphere
in our complex as a candidate for a ``natural powering'' process.
In §\ref{sec:Expansion-on-spheres} we show the following:
\begin{prop*}[Prop.\ \ref{prop:1-power-r-sphere} and \ref{prop:r-power-r-sphere}]
The $r$-spheres around a vertex in a two-dimensional $\At$-complex,
and even the $r$-powers (as a graph) of the $r$-spheres, do not
form a family of expanders.
\end{prop*}
The consequence is that any powering scheme in which the resulting
links are similar to the $r$-spheres, or to the $r$-powers of the
$r$-spheres in the original graphs, does not give a family of HD
expanders from $\At$-complexes (and in particular, from Ramanujan
complexes). This proposition also relates to a conjecture of Benjamini,
which states that there are no expander families in which the spheres
of any radius form a family of expanders themselves. Being excellent
local and global expanders, Ramanujan complexes are natural candidates
for disproving Benjamini's conjecture, but we show that they do not
violate it.

\subsection{Applications}

In §\ref{sec:Applications-to-sampling} we demonstrate applications
of the geodesic power operation of HD expanders. It is well known
that walks on expander graphs sample the vertices well, and a natural
question is whether longer walks along expanders can sample well short
walks. We use the power complex to design a long walk which samples
well short walks along geodesics. In this walk, two geodesics of length
$r$ are considered neighbors if they border a common triangle in
the geodesic $r$-power of the complex (see Figure \ref{fig:3-walk}).
The fact that the power complex is a HD expander (§\ref{subsec:Spectrum-of-links})
implies that this walk samples well the geodesics (Corollary \ref{cor:sample-long-walk}).

\begin{figure}[h]
\includegraphics[scale=0.3]{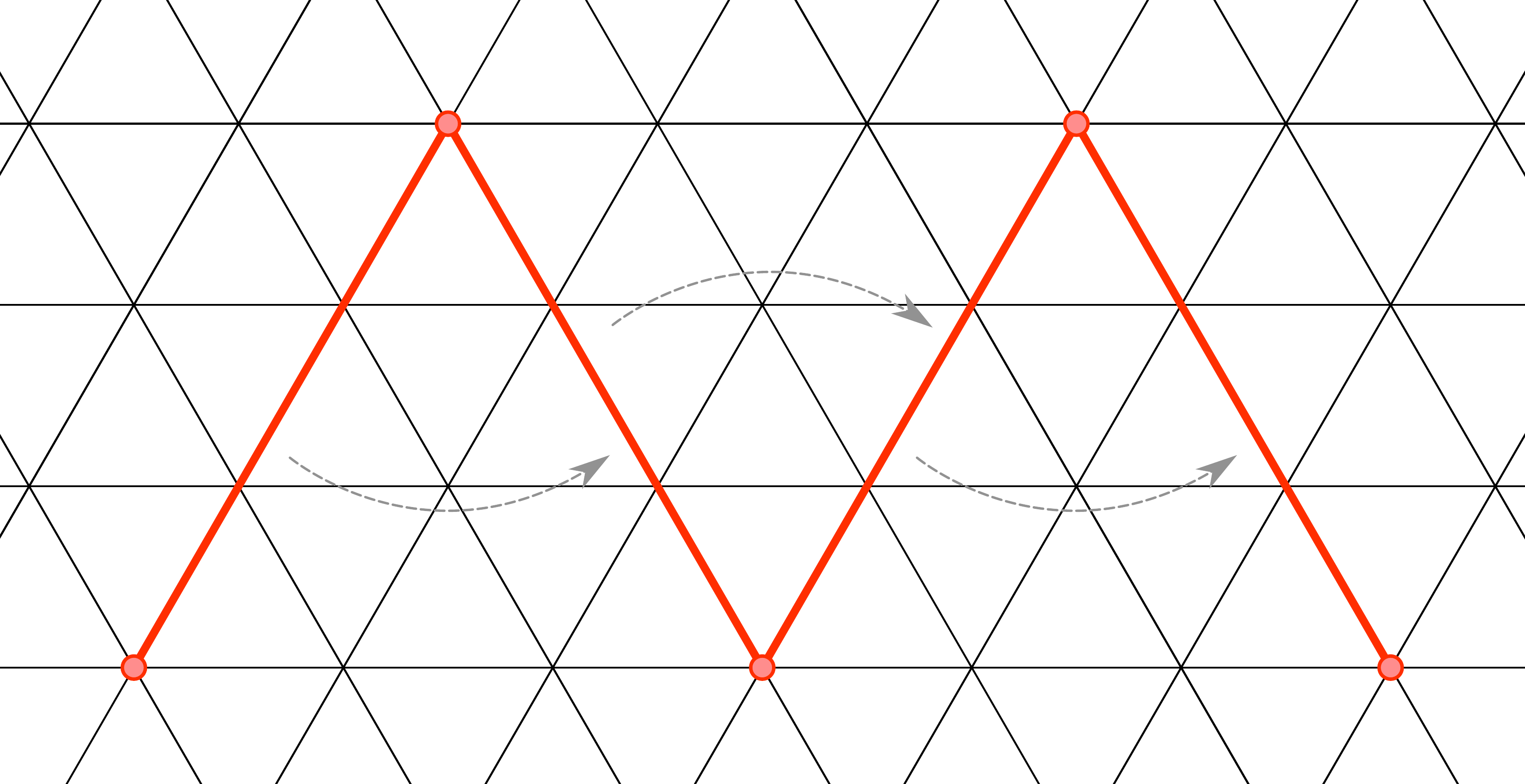}\caption{\label{fig:3-walk}Four steps of the $3$-walk on $3$-geodesics in
an $\widetilde{A}_{2}$-complex.}
\end{figure}

Combining this with the results of §\ref{subsec:Expansion-between-vertices-and-geod}
yields an application for computer science: an explicit construction
of a double sampler, as defined by Dinur and Kaufman for their work
on de-randomization of direct product testing \cite{Dinur2017Highdimensionalexpanders}.
The advantage of our construction over that of \cite{Dinur2017Highdimensionalexpanders}
is that the power operation allows us to obtain arbitrary sampling
quality for a fixed complex, in the same way that taking longer walks
along an expander graph improves the sampling quality in a classical
sampler. In contrast, the quality of the double sampler which appears
in \cite{Dinur2017Highdimensionalexpanders} is determined from the
underlying complex, and a new complex has to be generated each time
one seeks to obtain finer sampling quality.
\begin{acknowledgement*}
The authors thank David Kazhdan for helpful discussions, and the anonymous
referees for various improvements of the paper. Tali Kaufman was supported
by ERC and BSF grants; Ori Parzanchevski was supported by ISF grant
1031/17.
\end{acknowledgement*}

\section{\label{sec:Definitions}Definitions}

\subsection{\label{subsec:Simplicial-complexes}Simplicial complexes}

A simplicial complex $X$ with vertex set $V$ is a collection of
subsets of $V$, called faces or cells, which is closed under containment.
We denote by $X(j)$ the cells of size $j+1$, which are called $j$-dimensional
(or $j$-cells), and by $d=\dim X$ the maximal dimension of a cell;
$X$ is called \emph{pure} if every cell is contained in a $d$-cell.
The \emph{link} of a cell $\tau\in X$, denoted $X_{\tau}$, is the
complex obtained by taking all cells in $X$ that contain $\tau$
and removing $\tau$ from them. Thus, if $\dim\tau=i$ then $\dim X_{\tau}=d-i-1$,
and in particular $X_{\emptyset}=X$, and $X_{\tau}$ is a graph for
$\tau\in X(d-2)$. For any $\tau\in X(i)$ ($-1\leq i\leq d-2$),
the \emph{one-skeleton} of $X_{\tau}$ is the graph obtained by taking
only the vertices and edges in $X_{\tau}$, and the second largest
eigenvalue of the normalized adjacency operator of this graph is denoted
by $\mu_{\tau}$. If $X$ is maximal with respect to its underlying
graph, i.e.\ every clique in the one-skeleton of $X$ is a cell,
$X$ is called a \emph{clique complex}.
\begin{defn}[High dimensional expander]
\label{def:HDX}A pure $d$-dimensional complex $X$ is a \emph{$\lambda$-high
dimensional expander} if for every $-1\leq i\leq d-2$, and for every
$\tau\in X(i)$, $\mu_{\tau}\leq\lambda$. In fact, by \cite{Oppenheim2017Localspectral}
it follows that if $X$ is connected and $\mu_{\tau}\leq\lambda<\frac{1}{d}$
for every $\tau\in X(d-2)$, then $X$ is already a $\frac{\lambda}{1-\lambda d}$-high
dimensional expander.
\end{defn}

\subsection{\label{subsec:-Buildings-and-complexes}$\widetilde{A}$-complexes}

We recall the definition of Ramanujan complexes, and more generally
$\At$-complexes, following \cite{li2004ramanujan,Lubotzky2005a}
(for further study of these complexes see \cite{FGL+11,Lubotzky2013,kaufman2016isoperimetric,first2016ramanujan,Golubev2013triangle,Lubetzky2017RandomWalks}).
Let $F$ be a fixed non-archimedean local field with integer ring
$\mathcal{O}$, uniformiser $\pi$ and residue field $\mathcal{O}/\pi\mathcal{O}\cong\mathbb{F}_{q}$
(for example, $(F,\mathcal{O},\pi,q)=(\mathbb{Q}_{p},\mathbb{Z}_{p},p,p)$
with $p$ prime or $(\mathbb{F}_{q}((t)),\mathbb{F}_{q}[[t]],t,q)$
with $q$ a prime power). The \emph{affine Bruhat-Tits building} $\mathcal{B}=\mathcal{B}_{d}(F)$
\emph{of type $\At_{d-1}$} is an infinite $(d-1)$-dimensional clique
complex, whose vertices correspond to the cosets $\nicefrac{G}{K}$,
where $G=PGL_{d}(F)$ and $K=PGL_{d}(\mathcal{O})$. The map $gK\mapsto g\mathcal{O}^{d}$
constitutes a correspondence between $G/K$ and homothety-classes
of $\mathcal{O}$-lattices in $F^{d}$. Two such classes are neighbors
in $\mathcal{B}$ when they have representatives $L,L'$ satisfying
$\pi L<L'<L$, and $\mathcal{B}$ is the clique complex of the resulting
graph. An \emph{$\At_{d-1}$-complex} is, by definition, a quotient
of $\mathcal{B}$ by a cocompact lattice $\Gamma\leq G$ acting on
$\mathcal{B}$ without fixed points\footnote{By a theorem of Tits, if $\Gamma$ is torsion-free then this condition
is always satisfied.}, and we say that it is \emph{of density $q$}. The $\At_{d-1}$-complex
$X=\Gamma\backslash\mathcal{B}$ is a finite simplicial complex with
fundamental group $\Gamma$, and it is a $\frac{1}{\sqrt{q}}$-HD
expander by \cite{Gar73}. The vertices of $\mathcal{B}$ are colored
by
\[
\col\colon\mathcal{B}(0)\rightarrow\nicefrac{\mathbb{Z}}{d\mathbb{Z}},\quad\col\left(gK\right)=\ord_{q}\det g+d\mathbb{Z},
\]
and we say that \emph{$X$ is $d$-partite }if $\Gamma$ preserves
vertex-colors, namely, $\col$ factors through $X(0)=\Gamma\backslash\mathcal{B}(0)$.
The directed edges of $\mathcal{B}$, denoted $\mathcal{B}^{\pm}(1)$,
are also colored: 
\[
\col\colon\mathcal{B}^{\pm}(1)\rightarrow\left(\nicefrac{\mathbb{Z}}{d\mathbb{Z}}\right)^{\times},\quad\col\left(v\rightarrow w\right)=\col\left(w\right)-\col\left(v\right),
\]
and $\Gamma$ always preserves the color of edges, so that any $\At_{d-1}$-complex
$X$ inherits a coloring of its directed edges, $\col\colon X^{\pm}(1)\rightarrow\left(\nicefrac{\mathbb{Z}}{d\mathbb{Z}}\right)^{\times}$.
We give this property a name:
\begin{defn}[Colored complexes]
\label{def-colored-clique-complexes}A \emph{$d$-colored complex}
is a pure $\left(d-1\right)$-dimensional complex whose directed edges
are colored by $\left(\nicefrac{\mathbb{Z}}{d\mathbb{Z}}\right)^{\times}$,
so that each $(d-1)$-cell $\tau$ can be assigned vertex colors $\col_{\tau}:\tau\rightarrow\nicefrac{\mathbb{Z}}{d\mathbb{Z}}$
for which $\col(v\rightarrow w)=\col_{\tau}w-\col_{\tau}v$ for any
$v,w\in\tau$. We stress that $\col_{\tau}$ and $\col_{\tau'}$ need
not agree on $\tau\cap\tau'$, so $X$ need not be vertex-colored.
\end{defn}

We define the colored adjacency operator $A_{j}$ on $L^{2}\left(X(0)\right)$
for a $d$-colored complex $X$: 
\begin{equation}
\left(A_{j}f\right)\left(v\right)=\sum\nolimits _{{w\sim v\atop \col\left(v\rightarrow w\right)=j}}f\left(w\right).\label{eq:col-adj}
\end{equation}
In the case of $\mathcal{B}_{d}$ and its quotients, $A_{j}$ is regular
of degree  $\left[\begin{smallmatrix}d\\
j
\end{smallmatrix}\right]_{q}$, namely $\#\{w\,|\,{w\sim v\atop \col\left(v\rightarrow w\right)=j}\}=\left[\begin{smallmatrix}d\\
j
\end{smallmatrix}\right]_{q}$ for every vertex $v$. Here $\left[\begin{smallmatrix}d\\
j
\end{smallmatrix}\right]_{q}$ is the Gaussian binomial coefficient, and eigenvalues of $A_{j}$
of this magnitude (which account for periodicity, by Perron-Frobenius
theory) are said to be \emph{trivial}.
\begin{defn}[\cite{li2004ramanujan,Lubotzky2005a}]
An $\At_{d-1}$-complex $X$ is a \emph{Ramanujan complex} if for
$0<j<d$ every eigenvalue of $A_{j}$ is either trivial or contained
in the spectrum of $A_{j}$ acting on $L^{2}\left(\mathcal{B}_{d}(0)\right)$.
\end{defn}

\begin{rem}
Subsequent works \cite{first2016ramanujan,Lubetzky2017RandomWalks}
suggest stronger definitions for Ramanujan complexes, but for the
case $d=3$, which we use in §\ref{sec:Applications-to-sampling},
it is shown in \cite{kang2010zeta} that the various definitions agree.
\end{rem}

\begin{example}
When $d=3$, $\mathcal{B}_{d}$ is a triangle complex with constant
vertex degree $2\left(q^{2}+q+1\right)$ and edge degree $q+1$. The
degrees of $A_{1}$ and $A_{2}$ are both $q^{2}+q+1$, so that a
trivial eigenvalue satisfies $\left|\lambda\right|=q^{2}+q+1$. If
(and only if) $X$ is Ramanujan, the nontrivial eigenvalues satisfy
\[
\lambda\in\Spec\left(A_{j}\big|_{L^{2}\left(\mathcal{B}_{3}\right)}\right)=\left\{ q\left(z_{1}+z_{2}+z_{3}\right)\,\middle|\,{\left|z_{1}\right|=\left|z_{2}\right|=\left|z_{3}\right|=1\atop z_{1}\cdot z_{2}\cdot z_{3}=1}\right\} ,
\]
and in particular $\left|\lambda\right|\leq3q$.
\end{example}

\subsection{\label{subsec:Geodesic-powering}Geodesic powering}

The power operation which we define is based on geodesic paths in
colored complexes:
\begin{defn}[Geodesic path]
\label{def:geodesic}A sequence of vertices $v_{0},\ldots,v_{r}$
in a $d$-colored complex $X$ is called an \emph{$r$-geodesic path
of color one} if it is:
\begin{enumerate}
\item A non-backtracking path: $\{v_{i},v_{i+1}\}\in X(1)$ and $v_{i+2}\neq v_{i}$.
\item Geodesic: $\{v_{i},v_{i+1},v_{i+2}\}\notin X(2)$.
\item Of color one: $\col(v_{i}\rightarrow v_{i+1})=1$.
\end{enumerate}
Unless stated otherwise, by an \emph{$r$-geodesic} we always mean
an $r$-geodesic path of color one. Note that for $d\geq3$, (1) follows
from (3) since $\col v_{i+2}\neq\col v_{i}$. 
\end{defn}

The inverted path $v_{r},\ldots,v_{0}$ is a ``geodesic of color
$d-1$'', satisfying (1), (2), and $\col(v_{i}\rightarrow v_{i+1})=d-1$.
Geodesic paths of other colors can also be defined, but their geometric
intuition is less obvious and originates from higher-dimensional geometry\footnote{For an edge $v\rightarrow w$ of color $j$, a step in a geodesic
walk of color $j$ can be performed by completing it to a $j$-geodesic
path of color one $v=v_{0},v_{1},\ldots,v_{j}=w$ which is also a
$j$-cell, performing $j$ steps of the $j$-dimensional flow described
in \cite{Lubetzky2017RandomWalks,Chapman2019CutoffRamanujancomplexes}
on this (ordered) cell, and keeping the first and last vertex. }.

The geometric motivation for this definition is the following: when
walking on graphs, every edge gives rise to a ``trivial'' local
loop of length two, and the non-backtracking walk eliminates these
loops. In complexes of higher dimension, even the non-backtracking
walk has local loops formed by going around a triangle or a higher-dimensional
cell. The (monochromatic) geodesic walks avoid such loops -- indeed,
on the building itself there are no closed geodesic paths at all,
just as there are no closed non-backtracking paths on a tree (this
follows from being ``collision-free'' -- see \cite{Lubetzky2017RandomWalks}).

\begin{example}
\label{exa:standard-geod}With $F,\mathcal{O},\pi,q$ and $\mathcal{B}=\mathcal{B}_{d}(F)$
as in §\ref{subsec:-Buildings-and-complexes}, observe the vertices
$v_{i}=\diag\left(\pi^{i},1,\ldots,1\right)K$ in $\mathcal{B}\left(0\right)$
(which correspond to the homothety classes of the $\mathcal{O}$-lattices
$L_{i}=\pi^{i}\mathcal{O}\times\mathcal{O}^{d-1}$, respectively).
The path $v_{0}\rightarrow\ldots\rightarrow v_{r}$ is an $r$-geodesic
in $\mathcal{B},$ as $\pi L_{i}<L_{i+1}<L_{i}$, $\left[L_{i}:L_{i+1}\right]=q$
and no scaling $L'_{i+2}$ of $L_{i+2}$ satisfies $\pi L_{i}<L'_{i+2}<L_{i}$.
\end{example}

We can now define the geodesic power of a colored complex:
\begin{defn}[Geodesic powering]
For a pure $d$-dimensional colored complex $X$, the \emph{geodesic
$r$-power} of $X$ is defined as follows: its vertices are the same
as the vertices of $X$, and $d+1$ distinct vertices $v_{0},\ldots,v_{d}$
form a $d$-cell iff, possibly after reordering them, there is an
$r$-geodesic path (of color one) from each $v_{i}$ to $v_{i+1}$,
and from $v_{d}$ to $v_{0}$. The cells of lower dimension are the
subcells of the $d$-cells so defined.
\end{defn}

We do not assume here that $X$ is a clique complex, but the examples
studied in this paper are. We also remark that the vertices contained
in an edge in the  power-{}complex are only the two endpoints of the
corresponding $r$-geodesic in $X$ -- as in graph powering, the
interior vertices are ``forgotten''.

\subsection{\label{subsec:Explicit-example}Explicit example}

We give an explicit example of an $\At$-complex and its geodesic
$r$-powers, using the construction from \cite{Evra2018RamanujancomplexesGolden}.
Let $p,q$ be distinct primes equal to $1$ modulo $4$, let $p=\pi\overline{\pi}$
be a decomposition of $p$ in $\mathbb{Z}\left[i\right]$, and define
\[
S_{p}=\Bigg\{ s\in M_{3}\left(\mathbb{Z}\left[i\right]\right)\,\Bigg|\,{s^{*}s=pI,\ord_{\pi}\left(\det s\right)=1\atop s\equiv\left(\begin{smallmatrix}1 & * & *\\
* & 1 & *\\
* & * & 1
\end{smallmatrix}\right)\Mod{2+2i}}\Bigg\}.
\]
For example, taking $p=5$ and $\pi=2+i$ we have $S_{5}=\Big\{\left(\begin{smallmatrix}2i-1 & 0 & 0\\
0 & 2i-1 & 0\\
0 & 0 & -2i-1
\end{smallmatrix}\right)$, $\left(\begin{smallmatrix}2i-1 & 0 & 0\\
0 & 1 & 2\\
0 & -2 & 1
\end{smallmatrix}\right)$, $\left(\begin{smallmatrix}1 & 1+i & 1+i\\
1+i & 1 & -1-i\\
1+i & -1-i & 1
\end{smallmatrix}\right)$, $\ldots\Big\}$ (see \cite[Example 6.4]{Evra2018RamanujancomplexesGolden}
for the complete set). The set $S_{p}$ has $p^{2}+p+1$ elements,
and the directed Cayley graph spanned by it in $PGL_{3}(\mathbb{Q}\left[i\right])$
is isomorphic to the graph of color-one edges in the $\At_{2}$-building
$\mathcal{B}=\mathcal{B}_{3}(\mathbb{Q}_{p})$ (the edges of color
two are just the inverse edges). Fixing $\varepsilon=\sqrt{-1}\in\mathbb{F}_{q}$
and mapping $i\mapsto\varepsilon$ gives a ring homomorphism $\eta\colon\mathbb{Z}[i]\rightarrow\mathbb{F}_{q}$,
and we denote by $S_{p,q}$ the set of matrices in $PGL_{3}\left(\mathbb{F}_{q}\right)$
obtained from $S_{p}$ by applying $\eta$. The generated group $G=\left\langle S_{p,q}\right\rangle $
equals either $PSL_{3}\left(\mathbb{F}_{q}\right)$ or $PGL_{3}\left(\mathbb{F}_{q}\right)$,
and the directed Cayley graph $X^{p,q}=Cay\left(G,S_{p,q}\right)$
is the graph of color-one edges in a finite $\At_{2}$-complex of
density $p$.

Returning to $\mathcal{B}\cong Cay(\left\langle S_{p}\right\rangle ,S_{p})$
(where $\left\langle S_{p}\right\rangle \leq PGL_{3}(\mathbb{Q}[i])$),
for any $s,s'\in S_{p}$ either $e\rightarrow s\rightarrow ss'\rightarrow e$
is a triangle in $\mathcal{B}$, or $e\rightarrow s\rightarrow ss'$
is a geodesic. For each $s$ there are $p+1$ choices of $s'$ for
which the former occurs; they are the ones for which $ss's''$ is
a scalar matrix for some $s''\in S_{p}$. Denoting the remaining $p^{2}$
choices of $s'$ by $\Sigma_{s}$, the $r$-geodesics starting at
$e$ are precisely 
\[
e\rightarrow s_{1}\rightarrow s_{1}s_{2}\rightarrow\cdots\rightarrow s_{1}s_{2}\ldots s_{r}\qquad\left(s_{1}\in S_{p},s_{i}\in\Sigma_{s_{i-1}}\text{for }i\geq2\right).
\]
In accordance, the $r$-power of $\mathcal{B}$ coincides with the
Cayley graph with generating set $S_{p}^{(r)}:=\left\{ s_{1}s_{2}\ldots s_{r}\,\middle|\,s_{1}\in S_{p},s_{i}\in\Sigma_{s_{i-1}}\right\} \subseteq PGL_{3}(\mathbb{Q}[i])$,
and the $r$-power of $X^{p,q}$ is the Cayley graph of $G$ with
the generators $\big\{\eta\left(s\right)\,\big|\,s\in S_{p}^{(r)}\big\}$.

\section{Expansion in the power-complex}

\subsection{\label{subsec:Spectrum-of-links}Spectrum of links in the power-complex}

Let $F,\mathcal{O},\pi$ and $\mathbb{F}_{q}=\mathcal{O}/\pi\mathcal{O}$
be as in §\ref{subsec:-Buildings-and-complexes}. The link of a vertex
in $\mathcal{B}=\mathcal{B}_{d}\left(F\right)$ coincides with the
spherical building of $PGL_{d}(\mathbb{F}_{q})$, whose cells correspond
to flags in the space $\mathbb{F}_{q}^{d}$. Our first goal is to
give a similar description for the links in the power-complex of $\mathcal{B}$.
Recall that a module $M$ over a commutative ring $R$ is called \emph{free
}if it is isomorphic to $R^{\oplus m}$ for some $m$, which is denoted
by $\rank M$. The ring we are interested in is
\[
\mathcal{O}_{r}=\mathcal{O}_{r}\left(F\right):=\mathcal{O}/\pi^{r}\mathcal{O},
\]
and we remark that if $F$ is a completion of a global field $k$,
then one can also realize $\mathcal{O}_{r}$ as a quotient of the
integer ring of $k$; E.g., $\mathcal{O}_{r}(\mathbb{Q}_{p})=\mathbb{Z}_{p}/p^{r}\mathbb{Z}_{p}\cong\nicefrac{\mathbb{Z}}{p^{r}\mathbb{Z}}$,
and $\mathcal{O}_{r}(\mathbb{F}_{q}((t)))\cong\mathbb{F}_{q}\left[t\right]/\left(t^{r}\right)$.
We introduce the following definitions:
\begin{defn}
\label{def:free_proj}
\begin{enumerate}
\item A flag of $\mathcal{O}_{r}$-modules $\mathcal{F=}\left\{ 0<M_{1}<\ldots<M_{\ell}<\mathcal{O}_{r}^{d}\right\} $
is called \emph{free} if all $M_{i}$ are free $\mathcal{O}_{r}$-modules.
\item The \emph{free projective $d$-space }over $\mathcal{O}_{r}$, denoted
$\mathbb{P}_{\mathrm{fr}}^{d}(\mathcal{O}_{r})$, is the complex whose
vertices correspond to free $\mathcal{O}_{r}$-submodules $0<M<\mathcal{O}_{r}^{d+1}$,
and whose cells are the free flags in $\mathcal{O}_{r}^{d+1}$.
\end{enumerate}
\end{defn}

We state now a few useful facts which follow from the theory of modules
over local principal ideal rings (for example from the existence of
a Smith Normal Form over $\mathcal{O}_{r}$, see e.g.\ \cite{kaplansky1949elementary}):
\begin{fact}
\label{fact:Submodules-of-Ord}Every submodule $M$ of $\mathcal{O}_{r}^{d}$
is equivalent under $GL_{d}(\mathcal{O}_{r})$ to $\diag\left(\pi^{m_{1}},\ldots,\pi^{m_{d}}\right)\mathcal{O}_{r}^{d}$
for a unique choice of $r\geq m_{1}\geq\ldots\geq m_{d}\geq0$. For
these $m_{i}$, $\mathcal{O}_{r}^{d}/M\cong\diag\left(\pi^{r-m_{d}},\ldots,\pi^{r-m_{1}}\right)\mathcal{O}_{r}^{d}$,
and $M\leq\mathcal{O}_{r}^{d}$ is free iff all $m_{i}$ are either
$0$ or $r$. When $M$ is free, the free submodules in $\mathcal{O}_{r}^{d}/M$
are in correspondence with free submodules in $\mathcal{O}_{r}^{d}$
which contain $M$.
\end{fact}

An important consequence is that a maximal free flag $\mathcal{F}=\left\{ M_{i}\right\} $
in $\mathcal{O}_{r}^{d}$ has a unique refinement to a maximal flag,
since each quotient $\nicefrac{M_{i+1}}{M_{i}}$ is isomorphic to
the local ring $\mathcal{O}_{r}$, which has a unique composition
series.
\begin{example}
\label{exa:free-flag}For $F=\mathbb{Q}_{p},r=2,d=3$ we have $\mathcal{O}_{2}\cong\nicefrac{\mathbb{Z}}{p^{2}\mathbb{Z}}$,
and $0\negmedspace<\negmedspace\left(\begin{smallmatrix}*\\
0\\
0
\end{smallmatrix}\right)\negmedspace<\negmedspace\left(\begin{smallmatrix}*\\
*\\
0
\end{smallmatrix}\right)\negmedspace<\negmedspace\left(\nicefrac{\mathbb{Z}}{p^{2}\mathbb{Z}}\right)^{3}$ is a maximal free flag; its maximal refinement is $0\negmedspace<\negmedspace\left(\begin{smallmatrix}p*\\
0\\
0
\end{smallmatrix}\right)\negmedspace<\negmedspace\left(\begin{smallmatrix}*\\
0\\
0
\end{smallmatrix}\right)\negmedspace<\negmedspace\left(\begin{smallmatrix}*\\
p*\\
0
\end{smallmatrix}\right)\negmedspace<\negmedspace\left(\begin{smallmatrix}*\\
*\\
0
\end{smallmatrix}\right)\negmedspace<\negmedspace\left(\begin{smallmatrix}*\\
*\\
p*
\end{smallmatrix}\right)\negmedspace<\negmedspace\left(\nicefrac{\mathbb{Z}}{p^{2}\mathbb{Z}}\right)^{3}$.
\end{example}

To see how $\mathcal{O}_{r}^{d}$ relates to the building $\mathcal{B}$,
let $L_{0}\rightarrow\ldots\rightarrow L_{m}$ be a closed path of
color one in $\mathcal{B}$. By some abuse of notation we use $L_{i}$
to refer to a specific choice of lattice in the homothety class $L_{i}$,
that was chosen so that $\pi L_{i}<L_{i+1}<L_{i}$; it follows from
$\col(L_{i}\rightarrow L_{i+1})=1$ that $[L_{i}:L_{i+1}]=q$. By
$\col L_{0}=\col L_{m}$ one has $m=rd$ for some $r$, and from $L_{m}=\pi^{t}L_{0}$
(as they are homothetic) it follows that $t=r$, so that $L_{m}<L_{m-1}<\ldots<L_{0}$
projects to a maximal flag in $L_{0}/\pi^{r}L_{0}\cong\mathcal{O}_{r}^{d}$.
On the other hand, each maximal flag in $L_{0}/\pi^{r}L_{0}$ lifts
to a distinct path, since if $L_{i+1}\neq L_{i+1}'$ and both are
of index $q$ in $L_{i}$ then they cannot be homothetic. We conclude
that color-one cycles of length $rd$ around $L_{0}$ are in correspondence
with maximal flags in $\mathcal{O}_{r}^{d}$. In addition, as $\mathcal{B}$
is $d$-colored any color-one path $L_{0}\rightarrow\ldots\rightarrow L_{r}$
can be completed to a color-one cycle of length $rd$, so that color-one
paths of length $r$  starting from a given vertex correspond to flags
$M_{r}<\ldots<M_{0}=\mathcal{O}_{r}^{d}$ such that $[M_{i}:M_{i+1}]=q$.
We can now prove two useful Lemmas:
\begin{lem}
\label{lem:trans-geodes}The group $G=PGL_{d}(F)$ acts transitively
on all $r$-geodesics in $\mathcal{B}$.
\end{lem}

\begin{proof}
We show by induction on $r$ that any $r$-geodesic (of color one)
$L_{0}\dasharrow L_{r}$ can be translated by $G$ to the ``standard''
geodesic $L_{0}'\dasharrow L_{r}'$, where $L_{i}^{'}:=\pi^{i}\mathcal{O}\times\mathcal{O}^{d-1}$
as in Example \ref{exa:standard-geod}. For $r=0$ this holds by transitivity
of $G$ on $\mathcal{B}(0)=G/K$, and for $r=1$ since $K=Stab(\mathcal{O}^{d})$
acts transitively on the edges of color one leaving $\mathcal{O}^{d}$.
For $r\geq2$, we can assume by induction that $L_{i}=L_{i}'$ for
$0\leq i\leq r-1$, and we observe that the color one edges leaving
$L_{r-1}'=\pi^{r-1}\mathcal{O}\times\mathcal{O}^{d-1}$ enter either
\vspace{-1.5ex}
\begin{equation}
G_{\vec{a}}=\left(\begin{matrix}{\scriptscriptstyle \pi^{r}} & {\scriptscriptstyle a_{2}\pi^{r-1}\ \cdots\ a_{d}\pi^{r-1}}\\
0 & I_{d-1}
\end{matrix}\right)\mathcal{O}^{d},\qquad\text{or}\qquad T_{\vec{a}}=\left(\begin{smallmatrix}\pi^{r-1}\\
 & {\displaystyle I_{m}}\\
 &  & \vphantom{A}\pi & a_{m+2}\ \cdots\ a_{d}\\
 &  &  & {\displaystyle I_{d-m-2}}
\end{smallmatrix}\right)\mathcal{O}^{d}\label{eq:Ga_Ta}
\end{equation}
(with all $a_{i}\in\mathcal{O}$, and $0\leq m\leq d-2$). If $d>2$,
the vertices $T_{\vec{a}}$ are also neighbors of $L_{j-1}'$, and
if $d=2$ they coincide with it; in both cases they do not continue
$L_{j-1}'\rightarrow L_{j}'$ to a geodesic. On the other hand, $\bigcap_{i=0}^{r-1}Stab\left(L_{i}'\right)=\left(\begin{smallmatrix}{\scriptscriptstyle \mathcal{O}^{\!\times}} & {\scriptscriptstyle \text{---}\pi^{r-1}\mathcal{O}\text{---}}\\
{\scriptscriptstyle \stackrel[|]{|}{\mathcal{O}}} & GL_{d-1}\left(\mathcal{O}\right)
\end{smallmatrix}\right)$ takes $L_{r}'$ to all $G_{\vec{a}}$, which concludes the proof.
\end{proof}
\begin{lem}
\label{lem:geodesics}For a color-one path $\gamma=L_{0}\dasharrow L_{r}$
with $[L_{i}:L_{i+1}]=q$, the following are equivalent:
\begin{enumerate}
\item $L_{r}$ is a free submodule of $\nicefrac{L_{0}}{\pi^{r}L_{0}}\cong\mathcal{O}_{r}^{d}$.
\item $\gamma$ is the unique color-one path of length $r$ from $L_{0}$
to $L_{r}$.
\item $\gamma$ is geodesic.
\end{enumerate}
\end{lem}

\begin{proof}
\emph{(1)$\rightarrow$(2):} By Fact \ref{fact:Submodules-of-Ord},
if $L_{r}\leq\mathcal{O}_{r}^{d}$ is free so is $\mathcal{O}_{r}^{d}/L_{r}$,
and by index considerations, it is isomorphic to $\mathcal{O}_{r}$
which has a unique composition series. \emph{(2)$\rightarrow$(3):}
If $\gamma$ is not geodesic, then $\pi L_{i}\leq L_{i+2}<L_{i}$
for some $i$, and $[L_{i}:L_{i+2}]=q^{2}$. As $L_{i+2}$ corresponds
in $\nicefrac{L_{i}}{\pi L_{i}}\cong\mathbb{F}_{q}^{d}$ to a subspace
of codimension two, there are $(q+1)$ color-one paths $L_{i}\rightarrow*\rightarrow L_{i+2}$
(including $*=L_{i+1}$), hence $\gamma$ is not unique. \emph{(3)$\rightarrow$(1):}
By Lemma \ref{lem:trans-geodes}, some $g\in G$ takes $L_{0}\dasharrow L_{r}$
to $\mathcal{O}^{d}\dasharrow\pi^{r}\mathcal{O}\times\mathcal{O}^{d-1}$,
and $\pi^{r}\mathcal{O}\times\mathcal{O}^{d-1}$ corresponds to a
free submodule in $\mathcal{O}_{r}^{d}$.
\end{proof}
\begin{prop}
\label{prop:link-power}The link of a vertex in the $r$-power-complex
of $\mathcal{B}=\mathcal{B}_{d}$ is isomorphic to $\mathbb{P}_{\mathrm{fr}}^{d-1}(\mathcal{O}_{r})$,
and the link of a cell of codimension two is either a complete bipartite
graph, or isomorphic to the graph $\mathbb{P}_{\mathrm{fr}}^{2}(\mathcal{O}_{r})$.
\end{prop}

\begin{proof}
By definition, the $r$-power of $\mathcal{B}$ is defined by its
$\left(d-1\right)$-cells, which correspond to cycles of $\mathcal{O}$-lattices
$L_{0},L_{1},\ldots,L_{d-1},L_{0}\in\mathcal{B}(0)$ such that each
$L_{i},L_{i+1\Mod{d}}$ are connected by an $r$-geodesic. Such a
cycle can be rotated to start in any of its vertices, hence by Lemma
\ref{lem:geodesics} each $\left(d-1\right)$-cell containing $L_{0}$
in the power-complex corresponds to a unique color-one $rd$-cycle
\emph{$\left\{ v_{i}\right\} _{i=0}^{rd}$ }in $\mathcal{B}$ such
that $v_{ir}=L_{i}$. By the discussion following Example \ref{exa:free-flag},
this $rd$-cycle corresponds to a maximal flag $\left\{ M_{i}\right\} $
in $L_{0}/\pi^{r}L_{0}\cong\mathcal{O}_{r}^{d}$ which satisfies that
$M_{ir}/M_{\left(i+1\right)r}\cong\mathcal{O}_{r}$ for each $i$,
hence the sub-flag $\left\{ M_{ir}\right\} _{i=0}^{d}$ is a maximal
free flag in $\mathcal{O}_{r}^{d}$. On the other hand, each maximal
free flag in $L_{0}/\pi^{r}L_{0}\cong\mathcal{O}_{r}^{d}$ refines
to a unique maximal flag, which corresponds to a cycle consisting
of $d$ geodesics of length $r$, yielding a $(d-1)$-cell containing
$L_{0}$ in the power-complex. As the cells in $\mathbb{P}_{\mathrm{fr}}^{d-1}(\mathcal{O}_{r})$
are all the subsets of maximal free flags in $\mathcal{O}_{r}^{d}$,
this establishes the stated isomorphism.

If $\tau$ is a $(d-3)$-cell in the $r$-power of $\mathcal{B}$,
it corresponds to a free flag $\left\{ M_{i}\right\} $ of length
$d-2$ in $\mathcal{O}_{r}^{d}$, and two cases arise: either there
are two different $i$ such that $M_{i+1}/M_{i}\cong\mathcal{O}_{r}^{2}$,
or there is a single $i$ for which $M_{i+1}/M_{i}\cong\mathcal{O}_{r}^{3}$.
In the former case, the link of $\tau$ is a complete bipartite graph,
since the choices of the (free) refinements in the two places with
$M_{i+1}/M_{i}\cong\mathcal{O}_{r}^{2}$ are independent. In the latter
case, the possible refinements correspond precisely to maximal free
flags in $\mathcal{O}_{r}^{3}$, resulting in a complex isomorphic
to $\mathbb{P}_{\mathrm{fr}}^{2}(\mathcal{O}_{r})$.
\end{proof}
We can now prove that one-dimensional links in the power-complexes
are excellent expanders:
\begin{thm}
\label{thm:spec-r-links}The graph $\mathbb{P}_{\mathrm{fr}}^{2}(\mathcal{O}_{r})$
is a $\left(\left(q+1\right)q^{r-1}\right)$-regular connected bipartite
graph on $2\left(q^{2}+q+1\right)q^{2\left(r-1\right)}$ vertices,
with adjacency spectrum 
\[
\Spec\left(\mathbb{P}_{\mathrm{fr}}^{2}(\mathcal{O}_{r})\right)=\left\{ \pm\left(q+1\right)q^{r-1},\pm\sqrt{q^{2r-1}},\pm\sqrt{q^{2r-2}},\ldots,\pm\sqrt{q^{r+1}},\pm\sqrt{q^{r}}\right\} .
\]
In particular, its second-largest normalized eigenvalue equals $\frac{\sqrt{q}}{q+1}$
independently of $r$.
\end{thm}

\begin{proof}
Denoting by $\mathcal{F}_{r}^{i}$ the set of free modules of rank
$i$ in $\mathcal{O}_{r}^{3}$, the graph $\mathbb{P}_{\mathrm{fr}}^{2}(\mathcal{O}_{r})$
is the bipartite graph with vertices $\mathcal{F}_{r}^{1}\sqcup\mathcal{F}_{r}^{2}$
and edges given by inclusion. By Fact \ref{fact:Submodules-of-Ord},
$GL_{3}(\mathcal{O}_{r})$ acts transitively on each $\mathcal{F}_{r}^{i}$,
and the stabilizer of $\mathcal{O}_{r}^{3}\times0\times0\in\mathcal{F}_{r}^{1}$
is $\left\{ g=\left(\begin{smallmatrix}* & * & *\\
0 & * & *\\
0 & * & *
\end{smallmatrix}\right)\in GL_{3}(\mathcal{O}_{r})\right\} $, hence 
\begin{equation}
\left|\mathcal{F}_{r}^{1}\right|=\left|GL_{3}(\mathcal{O}_{r})\right|/\left|\left(\begin{smallmatrix}* & * & *\\
0 & * & *\\
0 & * & *
\end{smallmatrix}\right)\right|=\frac{\left(q^{3}-1\right)\left(q^{3}-q\right)\left(q^{3}-q^{2}\right)q^{9\left(r-1\right)}}{\left(q-1\right)q^{2}(q^{2}-1)(q^{2}-q)q^{7\left(r-1\right)}}=\left(q^{2}+q+1\right)q^{2\left(r-1\right)},\label{eq:F1r_size}
\end{equation}
and $|\mathcal{F}_{r}^{2}|=|\mathcal{F}_{r}^{1}|$ by a similar computation
or by duality (Fact \ref{fact:Submodules-of-Ord}). This gives the
size of $\mathbb{P}_{\mathrm{fr}}^{2}(\mathcal{O}_{r})$, and to comute
its degree we observe that the neighbors of a fixed vertex in $\mathcal{F}_{r}^{2}$
correspond to rank-one free submodules in $\mathcal{O}_{r}^{2}$,
of which there are
\[
\left|GL_{2}(\mathcal{O}_{r})\right|/\left|\left(\begin{smallmatrix}* & *\\
0 & *
\end{smallmatrix}\right)\right|=\frac{\left(q^{2}-1\right)\left(q^{2}-q\right)q^{4\left(r-1\right)}}{\left(q-1\right)q\left(q-1\right)q^{3\left(r-1\right)}}=\left(q+1\right)q^{\left(r-1\right)}.
\]
For any regular bipartite graph $\mathcal{G}$ on vertices $L\sqcup R$,
the spectrum of $A=\mathrm{Adj}_{\mathcal{G}}$ satisfies $\Spec A=\big\{\pm\!\sqrt{\lambda}\,\big|\,\lambda\in\Spec A^{2}\big|_{L}\big\}$.
We denote $\mathcal{A}=\mathrm{Adj}(\mathbb{P}_{\mathrm{fr}}^{2}(\mathcal{O}_{r}))$,
and let $Q=\mathcal{A}^{2}\big|_{\mathcal{F}_{r}^{1}}$. Recalling
that $\mathcal{F}_{r}^{1}$ corresponds to the endpoint of $r$-geodesics
leaving $v_{0}\in\mathcal{B}$, we observe the subgraph of $\mathcal{B}$
formed by these geodesic (including their inner vertices). This is
a rooted tree of height $r$, with the vertices $\mathcal{F}_{r}^{1}$
as leaves, root degree $\left(q^{2}+q+1\right)$, and all inner nodes
having $q^{2}$ descendants\footnote{Any color one edge has $q^{2}$ extensions to a geodesic of length
two - this follows from \eqref{eq:Ga_Ta}, \eqref{eq:F1r_size}.}. For $v,w\in\mathcal{F}_{r}^{1}$, denote by $\Delta\left(v,w\right)$
the shortest distance from $v$ and $w$ to a common ancestor in this
tree. This gives an ultrametric distance function on $\mathcal{F}_{r}^{1}$,
and the $\left(v,w\right)$-entries of all polynomials in $Q$ only
depend on $\Delta\left(v,w\right)$. For the rest of the proof all
matrices will be indexed by $\mathcal{F}_{r}^{1}$, and by $M_{\delta}$
we mean the value of $M_{v,w}$ for any $v,w$ with $\Delta\left(v,w\right)=\delta$.
Defining $B^{\left(\ell\right)}=\prod_{j=r}^{r+\ell-1}\left(Q-q^{j}\right)$,
we will show by induction that $B_{\delta}^{\left(\ell\right)}$ is
constant for $0\leq\delta\leq\ell$. In particular, $B^{\left(r\right)}=\prod_{j=r}^{2r-1}\left(Q-q^{j}\right)$
is a constant multiple of the all-one matrix, and this constant is
not zero as $Q\one=(q+1)^{2}q^{2(r-1)}\one$. This implies that $\mathbb{P}_{\mathrm{fr}}^{2}(\mathcal{O}_{r})$
is connected, and that the spectrum of $Q$ is $\big\{(q+1)^{2}q^{2(r-1)},q^{2r-1},q^{2r-2},\ldots,q^{r}\big\}$,
which yields the theorem. In fact, one can show by induction the following:
\[
B_{\delta-1}^{\left(\ell\right)}-B_{\delta}^{\left(\ell\right)}=\begin{cases}
q^{\ell r-\delta+{\ell-1 \choose 2}}\left(q^{\ell}-1\right)\prod_{j=\delta-\ell}^{\delta-2}\left(q^{j}-1\right) & \ell<\delta\leq r\\
0 & 1\leq\delta\leq\ell
\end{cases}
\]
(it turns out that this is not needed for the proof, but we record
this observation here for the benefit of future research). For two
matrices indexed by $\Delta$-values, the appropriate multiplication
rule is $\left(AB\right)_{\delta}=\sum_{\varepsilon,\zeta}N_{\varepsilon,\zeta}^{\delta}A_{\varepsilon}B_{\zeta}$,
where $N_{\varepsilon,\zeta}^{\delta}$ is the number of vertices
$u$ satisfying $\Delta\left(v,u\right)=\varepsilon$, $\Delta\left(u,w\right)=\zeta$,
for (any pair of) $v,w$ with $\Delta\left(v,w\right)=\delta$. It
is not hard to see that $N_{\varepsilon,\zeta}^{\delta}=N_{\zeta,\varepsilon}^{\delta}$,
so from now on we always assume $\varepsilon\leq\zeta$, writing 
\[
\left(AB\right)_{\delta}=\sum_{\varepsilon<\zeta}N_{\varepsilon,\zeta}^{\delta}\left(A_{\varepsilon}B_{\zeta}+A_{\zeta}B_{\varepsilon}\right)+\sum_{\varepsilon}N_{\varepsilon,\varepsilon}^{\delta}A_{\varepsilon}B_{\varepsilon}.
\]
By the ultrametric triangle inequality for $\Delta$, if $N_{\varepsilon,\zeta}^{\delta}\neq0$
then either $\delta<\varepsilon=\zeta$ or $\varepsilon\leq\zeta=\delta$,
so that we can further write 
\[
\left(AB\right)_{\delta}=\sum_{\varepsilon=0}^{\delta-1}N_{\varepsilon,\delta}^{\delta}\left(A_{\varepsilon}B_{\delta}+A_{\delta}B_{\varepsilon}\right)+\sum_{\varepsilon=\delta}^{r}N_{\varepsilon,\varepsilon}^{\delta}A_{\varepsilon}B_{\varepsilon}.
\]
Careful counting reveals that whenever $N_{\varepsilon,\zeta}^{\delta}\neq0$,
it is given by 
\begin{equation}
\begin{array}{c|c|c|c}
N_{\varepsilon,\zeta}^{\delta} & \varepsilon=0 & 0<\varepsilon<r & \varepsilon=r\vphantom{\Big|}\\
\hline \varepsilon\neq\delta & 1 & \left(q^{2}-1\right)q^{2\left(\varepsilon-1\right)} & \left(q^{2}+q\right)q^{2\left(r-1\right)}\vphantom{\Big|}\\
\varepsilon=\delta & 1 & \left(q^{2}-2\right)q^{2\left(\varepsilon-1\right)} & \left(q^{2}+q-1\right)q^{2\left(r-1\right)}\vphantom{\Big|}
\end{array}.\label{eq:N_d_e_z}
\end{equation}
In particular, whenever $\varepsilon<\delta$ or $\delta<\varepsilon<r$,
the value of $N_{\varepsilon,\zeta}^{\delta}$ does not depend on
$\delta$. This leads to many simplifications for the differences
between entries of $AB$, resulting in:
\begin{equation}
\begin{alignedat}{1}\left(AB\right)_{\delta-1}-\left(AB\right)_{\delta} & =\left(A_{\delta-1}-A_{\delta}\right)\sum_{\varepsilon<\delta-1}N_{\varepsilon,\delta}^{\delta}B_{\varepsilon}+\left(B_{\delta-1}-B_{\delta}\right)\sum_{\varepsilon<\delta-1}N_{\varepsilon,\delta}^{\delta}A_{\varepsilon}\\
 & \phantom{=}+N_{\delta-1,\delta-1}^{\delta-1}A_{\delta-1}B_{\delta-1}-N_{\delta-1,\delta}^{\delta}\left(A_{\delta-1}B_{\delta}+A_{\delta}B_{\delta-1}\right)+q^{2\left(\delta-1\right)}A_{\delta}B_{\delta}
\end{alignedat}
\label{eq:that}
\end{equation}
(note in particular that all $A_{t},B_{t}$ with $t>\delta$ cancel
out). We finally compute $Q$: given $L_{1},L_{1}'\in\mathcal{F}_{r}^{1}$
with $\Delta\left(L_{1},L_{1}'\right)=\delta$, the entry $Q_{\delta}$
is the number of $L_{2}\in\mathcal{F}_{r}^{2}$ which complete both
of them to a free flag. This corresponds to $L_{2}\leq\mathcal{O}_{r}^{3}/\left(L_{1}+L_{1}'\right)$
with $\mathcal{O}_{r}^{3}/L_{2}\cong\mathcal{O}_{r}$, giving 
\[
Q_{\delta}=\begin{cases}
\left(q+1\right)q^{r-1} & 0=\delta\\
q^{r-\delta} & 1\leq\delta\leq r
\end{cases}.
\]
We see that it is easier to work with $M:=Q-q^{r-1}$, which is simply
$M_{\delta}=q^{r-\delta}$, so that taking $A=M-q^{\ell+r}+q^{r-1}$
and $B=B^{\left(\ell\right)}$ we obtain $B^{\left(\ell+1\right)}=AB$,
and using \eqref{eq:that} we have
\begin{gather}
B_{\delta-1}^{\left(\ell+1\right)}-B_{\delta}^{\left(\ell+1\right)}=\left(AB\right)_{\delta-1}-\left(AB\right)_{\delta}=\left(MB\right)_{\delta-1}-\left(MB\right)_{\delta}-\left(q^{\ell+r}-q^{r-1}\right)\left(B_{\delta-1}-B_{\delta}\right)\nonumber \\
=\left[\begin{aligned} & \left(q^{r-\delta+1}-q^{r-\delta}\right)\sum\nolimits _{\varepsilon<\delta-1}N_{\varepsilon,\delta}^{\delta}B_{\varepsilon}+\left(B_{\delta-1}-B_{\delta}\right)\sum\nolimits _{\varepsilon<\delta-1}N_{\varepsilon,\delta}^{\delta}q^{r-\varepsilon}\vphantom{\Big|}\\
 & +N_{\delta-1,\delta-1}^{\delta-1}q^{r-\delta+1}B_{\delta-1}-N_{\delta-1,\delta}^{\delta}\left(q^{r-\delta+1}B_{\delta}+q^{r-\delta}B_{\delta-1}\right)\vphantom{\Big|}\\
 & +q^{r+\delta-2}B_{\delta}-\left(q^{\ell+r}-q^{r-1}\right)\left(B_{\delta-1}-B_{\delta}\right)\vphantom{\Big|}
\end{aligned}
\right]\nonumber \\
=\left[\begin{aligned} & \left(q^{r-\delta+1}-q^{r-\delta}\right)\sum\nolimits _{\varepsilon<\delta-1}N_{\varepsilon,\delta}^{\delta}B_{\varepsilon}+\left(B_{\delta-1}-B_{\delta}\right)\left(\sum\nolimits _{\varepsilon<\delta-1}N_{\varepsilon,\delta}^{\delta}q^{r-\varepsilon}-q^{\ell+r}+q^{r-1}\right)\vphantom{\Big|}\\
 & +B_{\delta-1}\left(N_{\delta-1,\delta-1}^{\delta-1}q^{r-\delta+1}-N_{\delta-1,\delta}^{\delta}q^{r-\delta}\right)-B_{\delta}\left(N_{\delta-1,\delta}^{\delta}q^{r-\delta+1}-q^{r+\delta-2}\right)
\end{aligned}
\right].\label{eq:monster}
\end{gather}
For $\delta\leq\ell$, we have by the induction hypothesis $B_{0}=\ldots=B_{\delta}$,
and thus \eqref{eq:monster} simplifies to 
\begin{equation}
q^{r-\delta}B_{0}\left[\left(q-1\right)\sum_{\varepsilon<\delta-1}N_{\varepsilon,\delta}^{\delta}+N_{\delta-1,\delta-1}^{\delta-1}q-N_{\delta-1,\delta}^{\delta}\left(q+1\right)+q^{2\left(\delta-1\right)}\right].\label{eq:this}
\end{equation}
For $1<\delta\leq\ell$, plugging \eqref{eq:N_d_e_z} into \eqref{eq:this}
gives 
\[
q^{r-\delta}B_{0}\left[(q-1)\Big(1+(q^{2}-1)\sum_{\varepsilon=1}^{\delta-2}q^{2(\varepsilon-1)})\Big)\negmedspace+(q^{2}-2)q^{2\delta-3}-(q^{2}-1)q^{2\delta-4}(q+1)+q^{2(\delta-1)}\right]=0,
\]
and for $\delta=1$, it gives $q^{r-1}B_{0}\left[q-\left(q+1\right)+1\right]=0$
as well. Finally, taking $\delta=\ell+1$ in \eqref{eq:monster} yields
\begin{gather*}
B_{0}\left[\left(q^{r-\ell}-q^{r-\ell-1}\right)\left(q^{2\ell-2}\right)-q^{\ell+r}+q^{\ell+r-1}+q^{\ell+r-2}+q^{2\ell-2}\left(\left(q^{2}-2\right)q^{r-\ell}-\left(q^{2}-1\right)q^{r-\ell-1}\right)\right]\\
+B_{\ell+1}\left[q^{\ell+r}-q^{\ell+r-1}-q^{\ell+r-2}-\left(q^{2}-1\right)q^{2\ell-2}\cdot q^{r-\ell}+q^{r+\ell-1}\right]=B_{0}\cdot0+B_{\ell+1}\cdot0=0,
\end{gather*}
which establishes the induction.
\end{proof}
In particular, we conclude that the link spectrum only depends on
the residue field:
\begin{cor}
\label{cor:isospect}If $F,F'$ are non-archimedean local fields with
the same residue order (e.g.\ $\mathbb{Q}_{p}$ and $\mathbb{F}_{p}((t))$),
then the free projective planes $\mathbb{P}_{\mathrm{fr}}^{2}(\mathcal{O}_{r})$
and $\mathbb{P}_{\mathrm{fr}}^{2}(\mathcal{O}_{r}')$ are isospectral
graphs.
\end{cor}

When $r=1$, both graphs coincide with the projective plane over the
residue field $\mathbb{F}_{q}$, and are isomorphic. However, we conjecture
that for $r\geq2$ these are non-isomorphic graphs when $F\not\cong F'$
(we have verified this for some small values of $r$ and $q$).

\subsection{\label{subsec:Expansion-between-vertices-and-geod}Expansion between
vertices and $r$-geodesics}

In this section we establish some global expansion results for geodesic
powers of Ramanujan complexes. These will be useful for the applications
we present in §\ref{sec:Applications-to-sampling}. Let $F$, $\mathcal{O}$,
and $q$ be as in §\ref{subsec:Spectrum-of-links}, $G=PGL_{3}(F)$,
$\Gamma<G$ a cocompact torsion-free lattice, and $X=\Gamma\backslash\mathcal{B}_{3}(F)$
a non-tripartite Ramanujan complex. Denote by $G^{(r)}=G^{(r)}(X)$
the bipartite graph formed by $r$-geodesics in $X$ on one side,
and the vertices $X(0)$ on the other one, with $v\in X(0)$ connected
to a geodesic $\gamma$ if it appears along it. In order to study
expansion in $G^{\left(r\right)}$ we first compute the expansion
across $m$-geodesics in $X$, for all $0\leq m\leq r$. Let $S_{m}^{1}\left(v\right)$
be all endpoints of $m$-geodesics (of color one) starting at $v$,
and $S_{m}^{2}\left(v\right)$ the set of $w$ with an $m$-geodesic
from $w$ to $v$ (or equivalently, a color-two $m$-geodesic from
$v$ to $w$). Denote
\[
\left(\mathcal{A}_{m}f\right)\left(v\right):=\sum\nolimits _{w\in S^{1}(v)\cup S^{2}\left(v\right)}f\left(w\right)\qquad\left(f\in L^{2}\left(X(0)\right)\right),
\]
and observe that $\mathcal{A}_{1}=A_{1}+A_{2}$, the colored adjacency
operators from \eqref{eq:col-adj}. We call an eigenvalue of $\mathcal{A}_{m}$
trivial if its eigenfunction is an eigenfunction of $A_{1},A_{2}$
with eigenvalue of magnitude $q^{2}+q+1$, and denote by $\lambda^{\left(m\right)}$
the largest nontrivial eigenvalue of $\mathcal{A}_{m}$ (in absolute
value).
\begin{thm}
\label{thm:spec-geodesics}Let $X$ be a Ramanujan $\widetilde{A}_{2}$-complex
of density $q$, and $m\geq1$. The degree of $\mathcal{A}_{m}$ on
$X$ is $k=2\left(q^{2}+q+1\right)q^{2\left(m-1\right)}$, and
\[
\lambda^{\left(m\right)}\leq\left(m^{2}+3m+2\right)q^{m}-2\left(m^{2}-1\right)q^{m-1}+\left(m^{2}-3m+2\right)q^{m-2},
\]
so that the largest normalized nontrivial eigenvalue of $\mathcal{A}_{m}$
is $\frac{\lambda^{\left(m\right)}}{k}\leq\frac{\left(m^{2}+3m+2\right)}{q^{m}}$.
The trivial spectrum of $\mathcal{A}_{m}$ (including multiplicities)
is
\[
{\textstyle {\text{\emph{trivial}}\atop \text{\emph{spectrum}}}}\!\left(\mathcal{A}_{m}\right)=\begin{cases}
\left\{ k\right\}  & X\text{ is non-tripartite}\\
\left\{ k,-\frac{k}{2},-\frac{k}{2}\right\}  & X\text{ is tripartite and \ensuremath{3\nmid m}}\\
\left\{ k,k,k\right\}  & X\text{ is tripartite and \ensuremath{3\mid m}}.
\end{cases}
\]
\end{thm}

\begin{proof}
The degree of $\mathcal{A}_{m}$ was already computed in the proof
of Theorem \ref{thm:spec-r-links}. We recall the notion of a \emph{Hecke
operator }on $\mathcal{B}$, which is a $G$-invariant operator $A$
on $L^{2}\left(\mathcal{B}(0)\right)$ such that $A\left(\one_{v}\right)$
has compact support for (any) $v\in V$. Such an operator induces
an action on quotients of $\mathcal{B}$, and we denote $A$ acting
on $X=\Gamma\backslash\mathcal{B}$ by $A\left(X\right)$. As $\Gamma$
is cocompact, $L^{2}\left(\Gamma\backslash G\right)$ decomposes as
a Hilbert sum of irreducible $G$-representations, $L^{2}\left(\Gamma\backslash G\right)=\widehat{\bigoplus}_{i}V_{i}$.
Letting $v_{0}=K\in\mathcal{B}(0)$, $\Gamma g\mapsto\Gamma gv_{0}$
gives an identification $X(0)\cong\Gamma\backslash G/K$ and thus
$L^{2}\left(X(0)\right)\cong L^{2}\left(\Gamma\backslash G\right)^{K}=\bigoplus_{i}V_{i}^{K}$,
where $V^{K}$ is the space of $K$-fixed vectors in $V$. It is well
known (see e.g.\ \cite{satake1966spherical,Lubotzky2005a} or \cite[§V]{macdonald1979symmetric})
that for every $i$ either $V_{i}^{K}=0$ or $V_{i}^{K}=\mathbb{C}\cdot f_{i}$,
where $f_{i}$ is a common eigenfunction of all Hecke operators on
$L^{2}\left(X(0)\right)$. By \cite{li2004ramanujan,Lubotzky2005a},
$X$ is a Ramanujan complex if each $V_{i}$ with $V_{i}^{K}\neq0$
is either finite-dimensional or \emph{tempered}, namely,\emph{ }the
matrix coefficient $\varphi_{i}\left(g\right)=\int_{\Gamma\backslash G}f_{i}\left(gx\right)\overline{f_{i}\left(x\right)}dx$
satisfies $\varphi_{i}\in\bigcap_{\varepsilon>0}L^{2+\varepsilon}\left(G\right)$.
The finite-dimensional representations account for the trivial eigenvalues,
so as $X$ is Ramanujan every nontrivial eigenvalue $\lambda_{i}$
of $\mathcal{A}_{m}$ comes from a tempered $V_{i}$. As $\varphi_{i}$
is a bi-$K$-spherical function on $G$, it can be interpreted as
a $K$-spherical function on $\mathcal{B}(0)$, which is also a $\lambda_{i}$-eigenfunction
for $\mathcal{A}_{m}$. In the language of \cite[§V.3]{macdonald1979symmetric}
its normalization $\frac{\varphi_{i}}{\varphi_{i}\left(v_{0}\right)}$
is a \emph{spherical function on $G$ relative to $K$}, and we denote
by $s_{1},s_{2},s_{3}$ its Satake parameters. Let $\mathcal{A}_{m}^{\pm}$
be the Hecke operators 
\[
\left(\mathcal{A}_{m}^{+}f\right)\left(v\right)=\sum\nolimits _{w\in S_{m}^{1}\left(v\right)}f\left(w\right),\qquad\left(\mathcal{A}_{m}^{-}f\right)\left(v\right)=\sum\nolimits _{w\in S_{m}^{2}\left(v\right)}f\left(w\right),
\]
and observe that $\mathcal{A}_{m}=\mathcal{A}_{m}^{+}+\mathcal{A}_{m}^{-}=\mathcal{A}_{m}^{+}+\left(\mathcal{A}_{m}^{+}\right)^{*}$.
The algebra of all Hecke operators is commutative, so that $\mathcal{A}_{m}^{+}$
commutes with $\left(\mathcal{A}_{m}^{+}\right)^{*}$, and denoting
by $\mu_{i}$ the $\mathcal{A}_{m}^{+}$-eigenvalue of $f_{i}$ (and
$\varphi_{i}$), we obtain $\lambda_{i}=\mu_{i}+\overline{\mu_{i}}$.
Under the identification of $L^{\infty}\left(\mathcal{B}(0)\right)$
with $L^{\infty}\left(G\right)^{K}$, the operator $\mathcal{A}_{m}^{+}$
is given by convolution with the characteristic function of $K\cdot\diag\left(\pi^{m},1,1\right)\cdot K$.
By \cite[§V.3]{macdonald1979symmetric}, $\mu_{i}$ is obtained by
specialization of the Hall-Littlewood polynomial corresponding to
the partition $m=m+0+0$:
\[
\mu_{i}=q^{m}P_{\left[m,0,0\right]}\left(s_{1},s_{2},s_{3};\tfrac{1}{q}\right)=q^{m}\sum_{\sigma\in S_{3}/S_{3}^{\left[m,0,0\right]}}\sigma\left(x_{1}^{m}\cdot\frac{x_{1}-x_{2}/q}{x_{1}-x_{2}}\cdot\frac{x_{1}-x_{3}/q}{x_{1}-x_{3}}\right)\bigg|_{x_{i}=s_{i}\ \left(i=1..3\right)};
\]
this uses \cite[§III.2 (2.2)]{macdonald1979symmetric}, where $S_{3}$
acts by permuting $\left\{ x_{1},x_{2},x_{3}\right\} $, and $S_{3}^{\left[m,0,0\right]}$
is the stabilizer of the partition $\left[m,0,0\right]$.  Being
tempered, $\frac{\varphi_{i}}{\varphi_{i}\left(v_{0}\right)}$ is
majorized by Harish-Chandra's $\Xi$-function (see e.g.\ \cite[Thm. 2]{Haagerup1988});
this is the spherical function with Satake parameters $s_{1}=s_{2}=s_{3}=1$,
and it is positive, so that 
\begin{align*}
\left|\mu_{i}\right| & =\left|\left(\mathcal{A}_{m}^{+}\frac{\varphi_{i}}{\varphi_{i}\left(v_{0}\right)}\right)\left(v_{0}\right)\right|\leq\sum_{w\in S_{m}^{+}\left(v_{0}\right)}\left|\frac{\varphi_{i}\left(w\right)}{\varphi_{i}\left(v_{0}\right)}\right|\leq\sum_{w\in S_{m}^{+}\left(v_{0}\right)}\Xi\left(w\right)=\left(\mathcal{A}_{m}^{+}\Xi\right)\left(v_{0}\right)\\
 & =q^{m}P_{\left[m,0,0\right]}\left(1,1,1;\tfrac{1}{q}\right).
\end{align*}
Direct computation then gives
\begin{multline*}
P_{\left[m,0,0\right]}\left(1,1,1;\tfrac{1}{q}\right)=\tfrac{1}{(x_{{1}}-1)(x_{{2}}-1)(x_{{1}}-x_{{2}})}\Big[((x_{{2}}^{2}-x_{{2}})x_{{1}}^{m}+x_{{1}}((1-x_{{1}})x_{{2}}^{m}+x_{{2}}(x_{{1}}-x_{{2}})))q^{m-2}\\
+((1-x_{{2}}^{2})x_{{1}}^{m+1}-(1-x_{{1}}^{2})x_{{2}}^{m+1}-x_{{1}}^{2}+x_{{2}}^{2})q^{m-1}+((x_{{2}}-1)x_{{1}}^{m+2}-(x_{{1}}-1)x_{{2}}^{m+2}+x_{{1}}-x_{{2}}){q}^{m}\Big]\bigg|_{x_{1}=x_{2}=1}\\
=\tfrac{1}{(x_{{1}}-1)^{2}}\Big[(x_{{1}}^{m}+(-m+1)x_{{1}}^{2}+(m-2)x_{{1}}){q}^{m-2}+(-2x_{{1}}^{m+1}+(m+1)x_{{1}}^{2}+1-m){q}^{m-1}\qquad\qquad\\
-{q}^{m}(-x_{{1}}^{m+2}+(m+2)x_{{1}}-m-1)\Big]\bigg|_{x_{1}=1}=\tfrac{1}{2}[(m^{2}+3m+2)q^{m}-2(m^{2}-1)q^{m-1}+(m^{2}-3m+2)q^{m-2}],
\end{multline*}
and the bound for $\lambda^{(m)}$ follows from $\lambda_{i}=\mu_{i}+\overline{\mu_{i}}$.
The finite-dimensional representations of $G$ are $\rho_{j}:g\mapsto\omega^{j\ord_{\pi}\det\left(g\right)}$
($\omega=e^{2\pi i/3}$), with $\rho_{0}$ being the trivial representation
(which appears once in $L^{2}\left(\Gamma\backslash G\right)$, and
$\rho_{1},\rho_{2}$ appearing in $L^{2}\left(\Gamma\backslash G\right)$
(each once) iff $X$ is tri-partite. The eigenvalue of $\mathcal{A}_{m}^{+}$
on $\rho_{j}$ can be computed by graph theory (observing that $\mathcal{A}_{m}^{+}$
shifts colors by $m$), or using its Satake parameters: $P_{\left[m,0,0\right]}\left(\frac{\omega^{j}}{q},\omega^{j},\omega^{j}q;\tfrac{1}{q}\right)=\omega^{jm}\left(q^{2}+q+1\right)q^{2\left(m-1\right)}$,
and the trivial eigenvalues of $\mathcal{A}_{m}$ are deduced as before.
\end{proof}
We return to the expansion of $G^{\left(r\right)}$:
\begin{prop}
\label{prop:vert-vs-geod}If $X$ is a non-tripartite Ramanujan $\At_{2}$-complex
of density $q$ then the Perron-Frobenius eigenvalue of $G^{(r)}=G^{(r)}(X)$
satisfies $\lambda_{1}\left(A_{G^{\left(r\right)}}\right)=rq^{r}\left(1+o\left(1\right)\right)$
(as $q,r\rightarrow\infty$), and its second eigenvalue satisfies
$\lambda_{2}\left(A_{G^{\left(r\right)}}\right)\leq\sqrt{r}q^{r}\left(1+o\left(1\right)\right)$.
\end{prop}

\begin{proof}
Since $G^{\left(r\right)}$ is bipartite, its nonzero spectrum is
obtained as $\left\{ \pm\sqrt{\lambda}\right\} $ where $\lambda$
runs over the nonzero eigenvalues of $A_{G^{\left(r\right)}}^{2}$
restricted to either side of $G^{(r)}$. For any vertex $v$, we have
\begin{align*}
\lambda_{1}\left(A_{G^{\left(r\right)}}^{2}\right) & =\#\left\{ {r\text{-geodesics}\atop \text{containing }v}\right\} \cdot\#\left\{ {\text{vertices \text{contained}}\atop \text{in an \ensuremath{r}-geodesic}}\right\} \\
 & =\left(r+1\right)\left(q^{2}+q+1\right)q^{2\left(r-1\right)}\cdot\left(r+1\right)=r^{2}q^{2r}\left(1+o(1)\right).
\end{align*}
Next, we observe that 
\[
N_{m}^{\left(r\right)}:=\#\left\{ {r\text{-geodesics containing}\atop v,w\text{ whenever \ensuremath{w\in S_{m}^{1}\left(v\right)}}}\right\} =\begin{cases}
\left(r+1\right)\left(q^{2}+q+1\right)q^{2\left(r-1\right)} & m=0\\
\left(r-m+1\right)q^{2\left(r-m\right)} & 0<m\leq r,
\end{cases}
\]
and that on the vertex side $A^{2}$ can be described using $\mathcal{A}_{m}$:
\[
A_{G^{\left(r\right)}}^{2}\big|_{{\text{vertex}\atop \text{side}}}=\sum\nolimits _{m=0}^{r}N_{m}^{\left(r\right)}\mathcal{A}_{m}.
\]
We observe that all $\mathcal{A}_{m}$ have a unique trivial eigenvalue,
obtained on the constant functions. This shows another way to compute
$\lambda_{1}\left(A_{G^{\left(r\right)}}^{2}\right)$, as $\sum_{m=0}^{r}N_{m}^{\left(r\right)}\deg(\mathcal{A}_{m})$.
More importantly, since all $\mathcal{A}_{m}$ are self-adjoint this
gives $\lambda_{2}(A_{G^{\left(r\right)}}^{2})\leq\sum_{m=0}^{r}N_{m}^{\left(r\right)}\lambda_{2}(\mathcal{A}_{m})$,
and we recall that $\lambda_{2}(\mathcal{A}_{m})\leq m^{2}q^{m}(1+o(1))$
for $m\geq1$ by Theorem \ref{thm:spec-geodesics}. In addition $\lambda_{2}(\mathcal{A}_{0})=\lambda_{2}\left(I\right)=1$,
so that 
\begin{alignat*}{2}
\lambda_{2}\left(A_{G^{\left(r\right)}}^{2}\right) & \leq\left(r+1\right)\left(q^{2}+q+1\right)q^{2\left(r-1\right)}\left(1+o(1)\right)\\
 & \ \ \ +\sum\nolimits _{m=1}^{r}\left(r-m+1\right)q^{2\left(r-m\right)}\cdot m^{2}q^{m}\left(1+o(1)\right)\  & =rq^{2r}\left(1+o(1)\right).\qedhere
\end{alignat*}
\end{proof}

\section{\label{sec:Expansion-on-spheres}Spheres in $\widetilde{A}$-complexes}

In this section we show that $r$-spheres in $\At$-complexes (of
a fixed degree) do not form a family of expanders, and neither do
their $r$-powers. This shows that any power operation on $\At$-complexes
whose links are similar to these spheres or to $r$-paths in them,
do not form a family of high-dimensional expanders. We carry out the
analysis for dimension two, but it is evident that similar phenomena
occur in general dimension.
\begin{prop}
\label{prop:1-power-r-sphere}The $r$-th spheres around a vertex
in $\mathcal{B}=\mathcal{B}_{3}(F)$ are not a family of expanders.
\end{prop}

\begin{proof}
This can be deduced from the spectral analysis in Proposition \ref{prop:r-power-r-sphere},
but we prefer to show how the geometry of the building gives an explicit
sparse cut in the $r$-sphere $S_{r}$ around a vertex. Let $F,\mathcal{O},\pi,q,G,K$
be as in §\ref{subsec:Expansion-between-vertices-and-geod}, and denote
\[
T=\left\{ \diag(\pi^{a},\pi^{b},\pi^{c})\,\middle|\,a,b,c\in\mathbb{N},\min\left(a,b,c\right)=0\right\} .
\]
The subcomplex induced by the vertices $\left\{ tv_{0}\,\middle|\,t\in T\right\} $
(where $v_{0}=K\in\mathcal{B}\left(0\right)$) is a triangular tiling
of the Euclidean plane, called the \emph{fundamental apartment} of
$\mathcal{B}$. There is a simplicial retraction from $\mathcal{B}$
to this apartment, which corresponds to a decomposition $G=\bigsqcup_{t\in T}BtK$;
here $B$ is the Iwahori group in $PGL_{3}\left(F\right)$, which
is the subgroup of elements in $K$ with subdiagonal entries in $\pi\mathcal{O}$.
In particular, each vertex in $\mathcal{B}$ lies in $X_{a,b,c}:=Btv_{0}$
for a unique $t=\diag(\pi^{a},\pi^{b},\pi^{c})\in T$. The $r$-sphere
around $v_{0}$ is the preimage of the $r$-sphere in the fundamental
apartment, which is a Euclidean hexagon: $S_{r}=\bigsqcup_{max(a,b,c)=r}X_{a,b,c}$
(see \cite[§3.2]{Evra2018RamanujancomplexesGolden}). 

The size of $X_{a,b,c}$ can be determined by computing Weyl lengths
\cite[§6.2]{Garrett1997}:
\[
\left|X_{a,b,c}\right|=\begin{cases}
q^{2\max\left(a,b,c\right)} & a\geq b\geq c\\
q^{2\max\left(a,b,c\right)-1} & a\geq c>b\text{ or }b>a\geq c\\
q^{2\max\left(a,b,c\right)-2} & b\geq c>a\text{ or }c>a\geq b\\
q^{2\max\left(a,b,c\right)-3} & c>b>a,
\end{cases}
\]
so that for $r\geq1$
\[
\left|S_{r}\right|=q^{2r-3}\left(qr+q+r-1\right)\left(q^{2}+q+1\right)\approx\left(r+1\right)q^{2r}.
\]
Finally, the degrees of vertices in $S_{r}$ (for $r\geq1$) are 
\[
\deg\left(v\in X_{a,b,c}\right)=\begin{cases}
q+1 & \left|\left\{ a,b,c\right\} \right|=2\\
2q & \left|\left\{ a,b,c\right\} \right|=3.
\end{cases}
\]
Assume for simplicity that $r$ is odd and larger than one (the computations
are similar in the even case), and let $A\subseteq S_{r}$ be the
half sphere
\[
A=\bigsqcup\left\{ X_{a,b,c}\,\middle|\,\max\left(a,b,c\right)=r\text{ and }\left({a\geq b\geq\frac{r+1}{2},\ b>a\geq c,\atop b\geq c>a,\text{ or }c>b\geq\frac{r+1}{2}}\right)\right\} .
\]
All edges crossing from $A$ to $S_{r}\backslash A$ connect either
$X_{r,\frac{r+1}{2},0}$ with $X_{r,\frac{r-1}{2},0}$, or $X_{0,\frac{r+1}{2},r}$
with $X_{0,\frac{r-1}{2},r}$. Each vertex in $X_{r,\frac{r+1}{2},0}$
has $q$ neighbors in $X_{r,\frac{r-1}{2},0}$ and $q$ neighbors
in $X_{r,\frac{r+3}{2},0}$, and similarly in the other case, giving
\[
\phi\left(S_{r}\right)\leq\frac{\left|E\left(A,S_{r}\backslash A\right)\right|}{\sum_{v\in A}\deg v}=\frac{q\left[|X_{r,\frac{r+1}{2},0}|+|X_{0,\frac{r+1}{2},r}|\right]}{r\left(q^{2}+q+1\right)\left(q+1\right)q^{2r-2}}=\frac{q^{2}-q+1}{q^{2}+q+1}\cdot\frac{1}{r}<\frac{1}{r}
\]
where $\phi$ is the graph conductance (also known as the normalized
Cheeger constant).
\end{proof}
While $S_{r}$ do not form an expander family as $r\rightarrow\infty$,
it is more interesting to ask whether the $r$-th power of $S_{r}$
(as a graph) form together such a family, since when we take the $r$-sphere
as an $r$-link, we should also expect edges in this link to correspond
to $r$-paths. Denoting by $\lambda_{\left(r\right)}$ the second
normalized eigenvalue of $S_{r}$, we have from the computation above
and the discrete Cheeger inequality that $\lambda_{\left(r\right)}\geq1-2\phi\left(S_{r}\right)>1-\frac{2}{r}$,
so that potentially we might have $\lambda_{\left(r\right)}^{r}\overset{{\scriptscriptstyle r\rightarrow\infty}}{\longrightarrow}e^{-2}<1$.
With a finer analysis we can rule out this possibility:
\begin{prop}
\label{prop:r-power-r-sphere}The normalized second eigenvalue $\lambda_{\left(r\right)}$
of the $r$-th sphere $S_{r}\subseteq\mathcal{B}_{3}$ satisfies
\[
\lambda_{\left(r\right)}\geq\cos\left(\frac{2\pi}{r}\right)=1-\frac{2\pi^{2}}{r^{2}}+O\left(\frac{1}{r^{4}}\right).
\]
In particular, $\lambda_{\left(r\right)}^{r}\overset{{\scriptscriptstyle r\rightarrow\infty}}{\longrightarrow}1$,
so the $r$-power graphs of the $r$-spheres in $\mathcal{B}_{3}$
are not expanders.
\end{prop}

\begin{proof}
Let $A$ be the adjacency operator on $S_{r}$, and $M$ its symmetric
normalization $M=D^{-1/2}AD^{-1/2}$ (where $D$ is the diagonal operator
of degrees in $S_{r}$). Let $f\colon S_{r}\rightarrow\mathbb{R}$
be the function 
\[
f\left(v\right)=\begin{cases}
\sin\left(\frac{2\pi j}{r}\right) & v\in X_{r,j,0}\text{ with }0\leq j\leq r\\
0 & \text{otherwise},
\end{cases}
\]
for which 
\[
\left\langle f,f\right\rangle =\sum_{j=1}^{r-1}\left|X_{r,j,0}\right|\sin\left(\tfrac{2\pi j}{r}\right)^{2}=q^{2r}\sum_{j=1}^{r-1}\sin\left(\tfrac{2\pi j}{r}\right)^{2}=\frac{rq^{2r}}{2}.
\]
Since for $0<j<r$ any $x\in X_{r,j,0}$ has degree $2q$ with $q$
neighbors in each of $X_{r,j-1,0}$ and $X_{r,j+1,0}$, and $f$ vanishes
elsewhere, we have $D^{-1/2}f=\frac{f}{\sqrt{2q}}$, and 
\begin{align*}
\left\langle Mf,f\right\rangle  & =\left\langle AD^{-1/2}f,D^{-1/2}f\right\rangle =\frac{1}{2q}\left\langle Af,f\right\rangle \\
 & =\frac{1}{2q}\sum_{j=1}^{r-1}\left|X_{r,j,0}\right|\left(q\sin\left(\tfrac{2\pi(j-1)}{r}\right)+q\sin\left(\tfrac{2\pi(j+1)}{r}\right)\right)\sin\left(\tfrac{2\pi j}{r}\right)=\frac{rq^{2r}}{2}\cos\left(\frac{2\pi}{r}\right).
\end{align*}
The involution $\tau:g\mapsto\left(\begin{smallmatrix} &  & 1\\
 & 1\\
1
\end{smallmatrix}\right)\left(g^{t}\right)^{-1}\left(\begin{smallmatrix} &  & 1\\
 & 1\\
1
\end{smallmatrix}\right)$ of $PGL_{3}$ induces an automorphism $\tau$ of $\mathcal{B}_{3}$
which restricts to $S_{r}$ and interchanges $X_{r,j,0}$ and $X_{r,r-j,0}$.
Since $f$ is $\tau$-antisymmetric (by construction) and the Perron-Frobenius
eigenvector of $M$ is $\tau$-symmetric (by connectedness of $S_{r}$),
they are orthogonal, hence $\lambda_{\left(r\right)}\geq\frac{\left\langle Mf,f\right\rangle }{\left\langle f,f\right\rangle }=\cos\left(\frac{2\pi}{r}\right)$.
\end{proof}
\begin{rem}
\begin{enumerate}
\item If one can show that the bound in Proposition \ref{prop:r-power-r-sphere}
is asymptotically tight, this would show that the $r^{2}$-powers
of the $r$-spheres in $\mathcal{B}$ form a family of expanders (but
with growing degrees).
\item For small values of $r$, the exact values of $\lambda_{\left(r\right)}$
are:
\end{enumerate}
\smallskip{}
\noindent\begin{minipage}[t]{1\columnwidth}%
\begin{center}
\begin{tabular}{cccc}
\toprule 
$r$ & 1 & 2 & 3\tabularnewline
\midrule
\midrule 
$\lambda_{\left(r\right)}$ & $\frac{\sqrt{q}}{q+1}$ & $\sqrt{\frac{1}{2}+\frac{\sqrt{q}}{2\left(q+1\right)}}$ & $\frac{\left(\left(q+1\right)\left(2i\sqrt{q^{3}+q^{2}+q}-q^{2}-1\right)\right)^{1/3}+q+1}{2\left(q+1\right)^{2/3}\left(2i\sqrt{q^{3}+q^{2}+q}-q^{2}-1\right)^{1/6}}$\tabularnewline
\midrule 
$\lim_{q\rightarrow\infty}\lambda_{\left(r\right)}$ & 0 & $\sqrt{\frac{1}{2}}$ & $\sqrt{\frac{3}{4}}$\tabularnewline
\bottomrule
\end{tabular}\medskip{}
\par\end{center}%
\end{minipage}\\
Finding $\lambda_{\left(r\right)}$ for general $r$ seems to be hard,
but determining $\lim_{q\rightarrow\infty}\lambda_{\left(r\right)}$
could be a nice problem.
\end{rem}

\section{\label{sec:Applications-to-sampling}Walks on geodesics and double
samplers}

In this section we describe a mixing random walk on the space of geodesics
in a Ramanujan complex, and use it to construct double samplers. We
first recall the definition of a sampler:
\begin{defn}[Sampler]
A connected bipartite incidence graph $G(L\sqcup R,E)$ with $L=[n]$,
$R\subseteq{[n] \choose k}$ (and $E=\{\left(\ell,r\right)|\ell\in r\}$)
is called an \emph{$f(\varepsilon,\alpha)$-sampler} if for any $S\subseteq L$
and $\varepsilon>0$ 
\[
\frac{1}{\left|R\right|}\left|\left\{ r\in R:\left|\frac{|r\cap S|}{k}-\frac{|S|}{|L|}\right|\geq\varepsilon\right\} \right|\leq\frac{1}{f(\varepsilon,|S|/|L|)}.
\]
\end{defn}

Namely, a random element in $R$ samples well any ``property'' $S$
which may be assigned to the elements of $L$. It is a classical result
that random walks on expanders sample well the vertices: Indeed, taking
$L$ to be the vertex set of a regular $\lambda$-expander (expander
with normalized nontrivial eigenvalues bounded by $\lambda$), and
$R$ the set of all paths of length $k$ in it, one obtains an $f(\varepsilon,\alpha)$-sampler
with 
\begin{equation}
f\left(\varepsilon,\alpha\right)=e^{\varepsilon^{2}k\left(1-\lambda\right)/60}\label{eq:expander-sampler}
\end{equation}
(for a proof take \cite[Thm.\ 3.2]{wigderson2005randomness} with
$f(v)=\one_{S}(v)-|S|/|L|$). A crucial point is that given a fixed
expander, one can improve the sampling precision by taking longer
and longer walks.

Double samplers were defined in \cite{Dinur2017Highdimensionalexpanders},
where they are used for studying PCP agreement tests and for a strong
de-randomization of direct products tests. Roughly, a double sampler
gives a way to sample well a set, and at the same time sample well
the sampling sets themselves. It it not known whether this can be
achieved from expander graphs - for example, whether long walks on
an expander graph (say of length $k^{2}$) sample short walks well
(say, of length $k$). 
\begin{defn}[Double sampler]
A tripartite incidence graph $G(L\sqcup R\sqcup W,E_{1}\sqcup E_{2})$
with $L=[n]$, $R\subseteq{[n] \choose k}$, $W\subseteq{[n] \choose K}$
(where $k\leq K$, $E_{1}=\{(\ell,r)|\ell\in r\}$ and $E_{2}=\{(r,w)|r\subseteq w\}$)
is called a \emph{$(f(\varepsilon,\alpha),f'(\varepsilon,\alpha))$-double-sampler}
if $G(L\sqcup R,E_{1})$ is an $f(\varepsilon,\alpha)$-sampler, and
$G(R\sqcup W,E_{2})$ is an $f'(\varepsilon,\alpha)$-sampler in the
sense that for any $T\subseteq R$
\[
\frac{1}{\left|W\right|}\left|\left\{ w\in W:\left|\frac{|\{r\in T|r\subseteq w\}|}{|\{r\in R|r\subseteq w\}|}-\frac{|T|}{|R|}\right|\geq\varepsilon\right\} \right|\leq\frac{1}{f'(\varepsilon,|T|/|R|)}.
\]
\end{defn}

Double samplers were constructed in \cite{Dinur2017Highdimensionalexpanders}
by taking $L$ to be the vertex set of a HD expander of dimension
$K-1$, $R$ to be the cells of dimension $d-1$ and $W$ the cells
of dimension $K-1$. The downside of this construction is that the
sampling quality of $L$ depends on the dimension of the complex,
and cannot be improved by taking longer walks as in the classic sampler
construction. 

\medskip{}

We propose here a new approach for the double sampling problem, by
designing a special walk on the space of geodesics $\left(d-1\right)$-cells
in an $\At_{d}$-complex. The upshot of our approach is that the quality
of the sampler depends on the length of the walks performed and not
on the dimension of the complex (which remains two dimensional). First
we introduce a walk which is interesting in its own right:
\begin{defn}
The \emph{$r$-walk} on an \emph{$\At_{d}$}-complex $X$ is the simple
random walk on the set of $(d-1)$-cells of the geodesic $r$-power
of $X$, where two cells are neighbors if they bound a joint $d$-cell
(in the power complex).
\end{defn}

Using the local-to-global technique, we obtain:
\begin{prop}
\label{prop:r-walk-expansion}The adjacency operator of the $r$-walk
on an $\widetilde{A}_{d}$-complex of density $q$ has normalized
nontrivial eigenvalues bounded by $\frac{d}{d+1}+\frac{d}{\sqrt{q}}$.
\end{prop}

\begin{proof}
By Proposition \ref{prop:link-power} and Theorem \ref{thm:spec-r-links},
the links of codimension two in the power complex are either complete
bipartite graphs or $\frac{\sqrt{q}}{q+1}$-expanders, and the claim
follows from \cite{Kaufman2017Highorderrandom}.
\end{proof}
This shows that the $r$-walk on an $\widetilde{A}_{d}$-complex of
density $q>d^{2}(d+1)^{2}$ samples well the geodesic $(d-1)$-cells
in it. The case which we will use for the double sampler construction
is that of $d=2$. There, the $r$-walk is carried on the (monochromatic)
geodesics of length $r$ in a two-dimensional complex $X$, and two
geodesics are neighbors if they share a joint triangle in the geodesic
$r$-power of $X$ (see Figure \ref{fig:3-walk}). In this way, $K/k$
steps of the $k$-walk yield a long walk (of length $K$) which samples
well the short walks (of length $k$) along geodesics. Indeed, applying
the classical results on expander samplers \eqref{eq:expander-sampler}
we obtain:
\begin{cor}
\label{cor:sample-long-walk}Let $X$ be an $\widetilde{A}_{2}$-complex
of density $q\geq37$. The incidence graph where $L$ are the $k$-geodesics
of $X$ and $R$ are the $k$-walks of length $K/k$ in $X$ is a
$\exp\left(\varepsilon^{2}\left(\frac{1}{3}-\frac{2}{\sqrt{q}}\right)\frac{K}{60k}\right)$-sampler.
\end{cor}

Combining this with the results of §\ref{subsec:Expansion-between-vertices-and-geod}
we arrive at a double sampler:
\begin{thm}
Let $X$ be a non-tripartite Ramanujan $\widetilde{A}_{2}$-complex
of density $q\geq37$. Taking $L$ to be the vertices of $X$, $R$
to be all $k$-geodesics in $X$, and $W$ to be all $k$-walks of
length $K/k$ in $X$, yields a $\left(\frac{\varepsilon^{2}k}{\alpha},\exp\left(\varepsilon^{2}\left(\frac{1}{3}-\frac{2}{\sqrt{q}}\right)\frac{K}{60k}\right)\right)$-double-sampler.
\end{thm}

\begin{proof}
Observe that $L\sqcup R$ is the graph $G^{(k)}$ of §\ref{subsec:Expansion-between-vertices-and-geod},
hence $\lambda_{1}(G^{(k)})\approx kq^{k}$, $\lambda_{2}(G^{(k)})\approx\sqrt{k}q^{k}$
(where $\approx$ stands for a multiplicative error of $\left(1+o\left(1\right)\right)$
as $k,q\rightarrow\infty$). Let $S\subset L$ be of size $\alpha|L|$
and let $T=\{r\in R:\frac{|r\cap S|}{k}\geq\alpha+\varepsilon\}$.
Using $|R|\approx q^{2r}|L|$ and the expander mixing lemma we obtain
\[
|T|k(\alpha+\varepsilon)\leq\left|E(S,T)\right|\leq\frac{\lambda_{1}|S||T|}{\sqrt{|R||L|}}+\lambda_{2}\sqrt{|S||T|}\approx k|T|\alpha+\sqrt{k}q^{k}\sqrt{\alpha|L||T|},
\]
so that $\frac{|T|}{|R|}\apprle\frac{\alpha}{\varepsilon^{2}k}$
as claimed. The expansion quality of $R\sqcup W$ is addressed in
Corollary \ref{cor:sample-long-walk}, with the difference that there
the incidence relation is of membership, and here it is of containment.
However, if $w=w_{0},\ldots,w_{K}$ is a $k$-walk of length $K/k$
(so that each $w_{kj},\ldots,w_{k\left(j+1\right)}$ is a $k$-geodesic),
it follows from the definition of the $k$-walk that for each $1\leq j\leq K/k-1$
the vertices $w_{kj-1},w_{kj},w_{kj+1}$ form a triangle. Thus, $w$
contains no other $k$-geodesics, and the two relations agree (and
in particular, $|\{r\in R|r\subseteq w\}|=K/k$). 
\end{proof}

\bibliographystyle{amsalpha}
\bibliography{/home/ori/Dropbox/Math/mybib}

\medskip{}

\noun{\small{}Department of Computer Science, Bar-Ilan University,
}\texttt{\small{}kaufmant@mit.edu}\noun{\small{}.}{\small\par}

\noun{\small{}Einstein Institute of Mathematics, Hebrew University,
}\texttt{\small{}parzan@math.huji.ac.il}\noun{\small{}.}{\small\par}
\end{document}